\documentclass[11pt,letterpaper]{article}

\usepackage[hmargin=0.9in,vmargin=0.95in]{geometry}
\usepackage{amsmath,amssymb,amsthm,mathrsfs}
\usepackage{cite}
\usepackage[hidelinks,bookmarksnumbered]{hyperref}

\numberwithin{equation}{section}
\allowdisplaybreaks[2]
\let\le\leqslant
\let\leq\leqslant
\let\ge\geqslant
\let\geq\geqslant

\theoremstyle{plain}
\newtheorem{theorem}{Theorem}[section]
\newtheorem{lemma}[theorem]{Lemma}
\newtheorem{corollary}[theorem]{Corollary}
\theoremstyle{definition}
\newtheorem{definition}[theorem]{Definition}
\theoremstyle{remark}
\newtheorem{remark}[theorem]{Remark}

\newcommand{\Case}[1]{\par\smallskip\noindent\textbf{#1}\quad\ignorespaces}

\hypersetup{
  pdftitle={Quantitative normal approximation for functionals of a Poisson occupation field in dimension 1+1},
  pdfauthor={Ao Huang, Guanglin Rang, Zhonggen Su}
}

\title{Quantitative normal approximation for functionals of a Poisson occupation field in dimension $1+1$}
\author{Ao Huang\textsuperscript{1} \quad Guanglin Rang\textsuperscript{2} \quad Zhonggen Su\textsuperscript{1}\\[0.6em]
  \small\textsuperscript{1}\,School of Mathematical Sciences, Zhejiang University, Hangzhou 310058, China\\
  \small\textsuperscript{2}\,School of Mathematics and Statistics, Wuhan University, Wuhan 430072, China\\[0.35em]
  \small\href{mailto:12435025@zju.edu.cn}{\texttt{12435025@zju.edu.cn}} \quad
  \href{mailto:glrang.math@whu.edu.cn}{\texttt{glrang.math@whu.edu.cn}} \quad
  \href{mailto:suzhonggen@zju.edu.cn}{\texttt{suzhonggen@zju.edu.cn}}}
\date{}

\begin{document}
\maketitle

\begin{abstract}
Let $\xi$ be the stationary occupation field generated by a Poisson system of independent simple symmetric random walks on $\mathbb Z$ in
space-time dimension $1+1$. For a finite nonempty set $A\subset\mathbb Z$, we consider the classical fixed-region observables
$W_N(A)$, the cumulative occupation of $A$ up to time $N$, and
$D_N(A)$, the number of distinct particles visiting $A$ up to time $N$. We prove quantitative central limit theorems for both observables, with optimal Wasserstein rates of order
$N^{-1/4}$. In addition, we introduce an independent nearest-neighbour
random walk $S=(S_n,\,n\ge 0)$ on $\mathbb Z$ with non-zero drift and sample the field along this ballistic
path. For a fixed polynomial observable $\varphi(x)=\sum_{j=0}^k \beta_j x^j, \beta_k\neq 0$, of degree $k\in \mathbb N$, consider the partial sums $Y_{N,\varphi}=\sum_{n=1}^N \varphi(\xi(n,S_n)).$
We prove a Wasserstein bound of order $N^{-1/2}$ for the normal approximation of the standardized
$Y_{N,\varphi}$. 
The drift induces an effective decorrelation of the sampled environment, leading to a substantial
improvement over fixed-region sampling. The proof
exploits a Poisson path-space representation of the occupation field together
with a conditional Malliavin--Stein method for Poisson functionals.

\end{abstract}

\section{Introduction and main results}\label{sec:introduction}
\noindent
{\bf Background.} We study quantitative normal approximation for functionals of a strongly correlated random field
in space-time dimension $1+1$. The field arises as the occupation field of a Poisson system of
independent simple symmetric random walks (SSRW) on $\mathbb Z$, started from an i.i.d.\ Poisson initial configuration.
Such Poisson systems go back at least to the work of Derman~\cite{Derman1955SomeCT},
and have since appeared in a variety of contexts as one of the simplest
examples of infinite particle systems; see, e.g.,
Liggett~\cite{liggett2013stochastic} and Kipnis and Landim~\cite{kipnis2013scaling} for background.

Besides being a natural Markovian model with an explicit invariant measure, this occupation field
has been used as a dynamic random environment in a number of works; see, e.g.,~\cite{hilario2015random,drewitz2011survival,gartner2006intermittency,Kesten20052402,MR753813}.
In particular, it has served as a random potential or catalyst through which other particles move;
see, e.g., Shen et al.~\cite{SHEN2021109066} and the references therein.
We now turn to the discrete-time model investigated in this paper.

\noindent
{\bf Model.} Throughout the paper we write $\mathbb N_0 = \{0,1,2,\dots\}$, $\mathbb N = \{1,2,3,\dots \}$ and let $\mathbb Z$ denote the set of all integers. We denote by $(\Omega,\mathcal F,\mathbb P)$ the underlying probability space supporting all random variables considered below, and by $\mathbb E$ the expectation. Let $Z\sim \mathcal N(0,1)$ be a standard Gaussian random variable.
The following definition is adapted from~\cite{MR753813,SHEN2021109066}.

\begin{definition}[Poisson field of independent walks]\label{def:pfiw}
At time $n=0$, we start with $\xi(0,x)$ particles at each site
$x\in\mathbb Z$, where the family $\{\xi(0,x),\,x\in\mathbb Z\}$ consists of
independent Poisson random variables with mean $\lambda>0$.
Each particle then performs an independent simple symmetric
random walk on $\mathbb Z$.
We denote by $\xi(n,x)$ the number of particles at position $x$
and time $n\in\mathbb N_0$.
More precisely,
\[
  \xi(n,x)
  =\sum_{y\in\mathbb Z}\sum_{i=1}^{\xi(0,y)}
    \mathbf{1}_{\{X_n^{y,i}=x\}},\qquad n\in\mathbb N_0,\ x\in\mathbb Z,
\]
where $X^{y,i}=(X^{y,i}_n,\,n\in\mathbb N_0)$ is the $i$-th random walk starting from $y$ at time~$0$.
\end{definition}

It is well known (see, for example,~\cite{kipnis2013scaling,Derman1955SomeCT})
that this infinite particle system admits the i.i.d.\ product Poisson
measure with mean $\lambda$ as an invariant and ergodic measure, so that
for each fixed $n\in\mathbb N_0$ the random variables $\{\xi(n,x),\,x\in\mathbb Z\}$
are again i.i.d.\ Poisson with mean $\lambda$.
In particular, the field
$\{\xi(n,x)\colon n\in\mathbb N_0,\,x\in\mathbb Z\}$ is stationary in time and homogeneous in space. On the other hand, it exhibits strong correlations in the time
direction: the value of $\xi(n,x)$ depends on the common history of many
particles and is therefore highly dependent across different $n$.
This strong temporal dependence is one of the main difficulties in establishing quantitative limit theorems.

In two earlier papers, Port~\cite{port1966system,port1967equilibrium} investigated the corresponding Poisson system of independent Markov chains and introduced, for a finite set $A\subset\mathbb Z$, two classical fixed-region observables. To make this more precise, let $A\subset\mathbb Z$ be a finite and nonempty set and define
\[
  \xi(n,A)=\sum_{x\in A}\xi(n,x),\qquad n\in\mathbb N_0,
\]
the total number of particles in $A$ at time $n$.
In our notation, his cumulative occupation functional is defined by
\[
  W_N(A)=\sum_{n=1}^N \xi(n,A)
  =\sum_{n=1}^N\sum_{x\in A}\xi(n,x),\qquad N\in\mathbb N.
\]
He also considered the number of distinct
particles that visit $A$ up to time $N$,
\[
  D_N(A):=\sum_{y\in\mathbb Z}\sum_{i=1}^{\xi(0,y)}
    \mathbf{1}_{\{\exists\, 1 \le m\le N:\ X_m^{y,i}\in A\}}.
\]
By definition, each particle is counted at most once in $D_N(A)$, regardless of
how many times it visits $A$ during the time interval $\{1,\dots,N\}$.
Port proved strong laws of large numbers and central limit theorems for $W_N(A)$ and 
$D_N(A)$. Following Port’s work, Cox and Griffeath~\cite{MR753813} studied continuous-time analogues
and established large deviation principles for occupation-time type functionals of Poisson systems
of independent random walks on $\mathbb Z^d$. However, these results are essentially qualitative and do not provide explicit rates of convergence.
Moreover, the observables considered there involve only finitely many spatial sites.

\noindent
{\bf Main results.} In the present work, we revisit the same Poisson field and focus on normal approximation with explicit rates. We consider functionals that depend on a much larger collection of space-time points, including functionals sampled along an independent random walk path.

Our first main result gives Wasserstein bounds for the classical fixed-region observables
$W_N(A)$ and $D_N(A)$. We write $d_W$ for the Wasserstein distance; see
Section~\ref{sec:MSB} for the definition.
\begin{theorem}[Quantitative CLT for fixed-region functionals]
\label{thm:WN}
Let
\[
  F_N^{(A)}=\frac{W_N(A)-\mathbb E W_N(A)}
  {\sqrt{\operatorname{Var}(W_N(A))}},
  \qquad
  G_N^{(A)}=\frac{D_N(A)-\mathbb E D_N(A)}
  {\sqrt{\operatorname{Var}(D_N(A))}}.
\]
Then there exist constants $0<c_i\le C_i<\infty$, $i=1,2$, depending only on $\lambda$ and $A$, such that for all
$N\ge 1$,

{\rm(i)}
\[
c_1N^{-1/4}\le d_W\bigl(F_N^{(A)},Z\bigr)\le C_1N^{-1/4},
\]

{\rm(ii)}
\[
  c_2N^{-1/4}\le d_W\bigl(G_N^{(A)},Z\bigr)\le C_2N^{-1/4}.
\]
In particular, the order $N^{-1/4}$ is optimal for both Wasserstein distances.
\end{theorem}

In order to go beyond a fixed finite region $A$, we introduce an
independent biased (simple asymmetric) nearest-neighbour random walk
$S=(S_n,n\in\mathbb N_0)$ on $\mathbb Z$, independent of the Poisson field $\xi$.
More precisely, $S_0=0$ and its increments $\Delta_n:=S_n-S_{n-1}$ are i.i.d. with
\[
\mathbb P_S(\Delta_n=+1)=\mathfrak p,\qquad \mathbb P_S(\Delta_n=-1)=\mathfrak q:=1-\mathfrak p,
\qquad \mathfrak p\in(0,1),\ \mathfrak p\neq\tfrac12,
\]
so that $S$ has non-zero drift $v:=\mathbb E_S[\Delta_1]=\mathfrak p-\mathfrak q\neq0$ and is ballistic.
Sampling along this random path allows us to explore the random environment at
spatial locations that change with time and typically separate linearly in time.

Motivated in part by~\cite{SHEN2021109066}, we study polynomial functionals sampled along the
trajectory.
For a fixed polynomial observable
\[
  \varphi:\mathbb N_0\to\mathbb R,
  \qquad
  \varphi(x)=\sum_{j=0}^k \beta_jx^j,
  \qquad
  \beta_k\neq 0,
\]
of degree $k\in \mathbb N$, we study the partial sums
\begin{equation}\label{eq:Yk-def-intro}
  Y_{N,\varphi}:=\sum_{n=1}^N \varphi(\xi(n,S_n)).
\end{equation}
From the perspective of random media, the case $\varphi(x)=x$ corresponds
to the cumulative potential sampled along the trajectory.
More general polynomial observables naturally encode nonlinear fluctuation
statistics and appear, for instance, in perturbative expansions of partition
functions for polymers in such environments.

We write
\[
  P_n(x)=\mathbb P_S(S_n=x),\qquad x\in\mathbb Z,\ n\in\mathbb N_0,
\]
for the $n$-step transition probabilities of $S$.
More generally, if $S_0=y$,
\[
  \mathbb P_S(S_n=x\mid S_0=y)=P_n(x-y),\qquad x,y\in\mathbb Z.
\]
Conditionally on the initial configuration $\{\xi(0,x),\,x\in\mathbb Z\}$, the
family of particle walks
\[
  \{X^{y,i}\colon y\in\mathbb Z,\ 1\le i\le \xi(0,y)\}
\]
is independent, and each $X^{y,i}=(X^{y,i}_n,n\in\mathbb N_0)$ is a simple
symmetric random walk on $\mathbb Z$ starting from $y$.
We denote by
\[
  Q_n(x)=\mathbb P\big(X^{0,i}_n = x \big),\qquad x\in\mathbb Z,\ n\in\mathbb N_0,
\]
the $n$-step transition probabilities of the particle walks. Hence, for a
general starting point $y$,
\[
  \mathbb P\big(X^{y,i}_n = x \,\big|\, X^{y,i}_0 = y\big)
  = Q_n(x-y),\qquad x,y\in\mathbb Z.
\]
In our setting, $Q_n$ corresponds to the SSRW kernel
governing the particle trajectories in the Poisson field, whereas $P_n$ is the
transition kernel of the biased sampling walk $S$.

Our second main theorem concerns the functionals sampled along the random
walk path~\eqref{eq:Yk-def-intro}.

\begin{theorem}[Annealed normal approximation for $Y_{N,\varphi}$ under drifted sampling]\label{thm:Yk-main}
For every fixed polynomial $\varphi$, define the normalized functional \[ F_{N,\varphi}=\frac{Y_{N,\varphi}-\mathbb E[Y_{N,\varphi}]}{\sigma_{N,\varphi}},\qquad\sigma_{N,\varphi}^2=\operatorname{Var}(Y_{N,\varphi}). \]
Assume that the sampling walk has non-zero drift $v\neq 0$. 
Then there exists a constant $C\in(0,\infty)$,
depending on $\varphi, \lambda, \mathfrak p$, such that for all $N\ge1$,
\[
  d_W\!\big(F_{N,\varphi},Z\big)\le C\,N^{-1/2}.
\]
Moreover, there exists a constant $\mathfrak c_\varphi\in(0,\infty)$
such that
\[
  \sigma_{N,\varphi}^2=\mathfrak c_\varphi\,N+O(1),
  \qquad N\to\infty.
\]
More precisely,
\[
  \mathfrak c_\varphi
  =
  \sum_{q=1}^k c_{\varphi,q}^2\,q!\,\lambda^q
  \Bigl(1+2\sum_{t\ge1}a_t^{(q)}\Bigr),
  \qquad
  a_t^{(q)}:=\mathbb E_S\!\big[Q_t(S_t)^q\big],
\]
where $(c_{\varphi,q})_{q\ge1}$ are the Poisson--Charlier coefficients of the centered
single-site observable
  $\varphi(\xi(n,x))-\mathbb E[\varphi(\xi(n,x))]$
as in Lemma~\ref{lem:single-site-k-chaos}. 
\end{theorem}

\begin{remark}\label{rem:rates}
\begin{enumerate}
\item \emph{Fixed-region sampling is slower.}
For a fixed finite set $A\subset\mathbb Z$, the lag-$t$ covariance
$\operatorname{Cov}\bigl(\xi(0,A),\xi(t,A)\bigr)$
decays at the same polynomial rate as $Q_t(0)$ in dimension one. This long memory is reflected
in the slower Wasserstein rate $N^{-1/4}$ in Theorem~\ref{thm:WN}~(i).

\item \emph{Ballistic path sampling restores the classical rate.}
Along the drifted walk $S$, the relevant coefficients are
$a_t^{(q)}=\mathbb E_S[Q_t(S_t)^q].$
Ballistic separation and heat-kernel bounds imply that these coefficients are summable, yielding
linear variance growth and the rate $N^{-1/2}$ in Theorem~\ref{thm:Yk-main}.
\end{enumerate}
\end{remark}

Theorems~\ref{thm:WN}~(i) and~\ref{thm:Yk-main} thus exhibit a clear dichotomy between fixed-region
sampling and ballistic sampling. While $W_N(A)$ retains the long temporal memory of the occupation
field in one dimension, sampling along a drifted path effectively decorrelates the environment and
yields the typical $N^{-1/2}$ Wasserstein rate for the entire family of polynomial observables
$Y_{N,\varphi}$.

\noindent
{\bf Related work.}
We briefly distinguish the present setting from several related lines of work. First,
Jaramillo and Murillo-Salas~\cite{jaramillo2025asymptotics} obtain quantitative normal
approximation for additive functionals of Poisson particle systems of the form
\[ A_r[\psi]=\int_0^r \langle \mu_t^\eta,\psi\rangle\,\mathrm dt, \qquad Z_r[\psi]=\frac{A_r[\psi]-\mathbb E[A_r[\psi]]}{\sigma_r}. \] In particular, for one-dimensional uniformly elliptic diffusions and $\psi\ge0$, their Proposition~5.2 yields a Wasserstein upper bound of order $r^{-1/4}$ for $Z_r[\psi]$. This agrees with the $N^{-1/4}$ scale in Theorem~\ref{thm:WN}~(i) for the discrete-time occupation functional $W_N(A)$. $W_N(A)$ may be viewed as a lattice analogue of such occupation-time additive functionals (formally with $\psi=\mathbf 1_A$ and $r\approx N$). In the present discrete setting, we further obtain matching Wasserstein lower bounds, showing that the $N^{-1/4}$ rates for both $W_N(A)$ and $D_N(A)$ are optimal in order. However, their
framework does not directly cover either the distinct-visitor count $D_N(A)$ or the
path-sampled polynomial functional $Y_{N,\varphi}$.

Second, CLTs and quantitative asymptotic results in dynamic random environments typically concern the position of the walker
itself; see, e.g.,
\cite{hilario2015random,redig2013transference,boldrighini2015continuous}. In contrast, we focus on the environment
functional
$\sum_{n=1}^N\varphi(\xi(n,S_n)),$
rather than the position of an environment-driven walker.

Third, from a methodological viewpoint, a key ingredient is that the
occupation field $\xi$ admits a representation as a functional of a Poisson
random measure on the path space. This representation places our problem within
the general theory of Poisson functionals, where Malliavin--Stein methods
provide quantitative Gaussian bounds for Poisson $U$-statistics and finite
Poisson-chaos expansions; see, e.g.,
\cite{peccati2010stein,reitzner2013central,eichelsbacher2014new,
last2016normal,dobler2018fourth}.
Since $Y_{N,\varphi}$ also depends on the independent sampling path $S$, our setting is not a direct application of these general Poisson
bounds. The distinctive feature of our approach is to combine the
Poisson path-space representation of the occupation field with a conditional
Malliavin--Stein reduction; see Section~\ref{sec:Yk} for details.

Against this background, to the best of our knowledge, quantitative normal
approximation has not previously been established for polynomial observables
of a Poisson occupation field sampled along an independent random walk path.

\noindent
{\bf Proof strategy.}
Building on the Poisson path-space representation and the
Malliavin--Stein framework discussed above, our argument proceeds in three
steps.

\noindent
{\rm(i)} We use the path-space representation \eqref{eq:xiFunctional}.
For $W_N(A)$ and $D_N(A)$, this yields the first-chaos representations
\eqref{eq:WN-proof-chaos} and \eqref{eq:DN-proof-chaos}, respectively.
For $Y_{N,\varphi}$, conditioning on the sampling path yields a finite
Poisson-chaos expansion; see \eqref{eq:Yk-chaos}.
      
\noindent
{\rm(ii)} We determine the variance scales from the corresponding
covariance sums. In the fixed-region setting, $\operatorname{Var}(W_N(A))\asymp N^{3/2}$ and $\operatorname{Var}(D_N(A))\asymp N^{1/2}$, whereas drifted path sampling yields linear variance growth through
      the summability of the averaged correlations $a_t^{(q)}$; see \eqref{eq:ck-constant}.
      
\noindent
{\rm(iii)} We apply Poisson Malliavin--Stein bounds. For $Y_{N,\varphi}$, this requires an annealed reduction obtained by conditioning on the sampling path; see \eqref{eq:MS-Yk}. The main technical input
      is then the control of the associated contraction kernels and fluctuation terms; see
      Lemma~\ref{lem:Hrs-L2-Yk}.

\noindent
{\bf Discussions.}
{\rm(i)} \emph{Scope of the results.}
All results in this paper are proved for the one-dimensional discrete-time system of SSRW particles.
Extensions to higher dimensions, continuous time, or other particle motions would require separate
verification of the relevant heat-kernel and local-limit estimates, the variance and range or hitting
asymptotics for the fixed-region observables, and the summability and contraction estimates for the
path-sampled observables. We do not claim such extensions here and keep the paper focused.

\noindent
{\rm(ii)} \emph{Role of the drift.}
For the path-sampled observables, the role of the drift is structural rather than cosmetic. In the
drifted regime, the averaged correlations are summable and the variance is asymptotically linear,
which is the mechanism behind the $N^{-1/2}$ rate in Theorem~\ref{thm:Yk-main}. By contrast, under
symmetric sampling the correlation decay is only polynomial. As shown in Appendix~\ref{app:symmetric-sampling},
the variance growth then depends on the first non-vanishing Poisson-chaos coefficient of the
observable. The appendix also explains why the conditional-variance step used in the drifted
Malliavin--Stein argument no longer closes in the rank-one symmetric regime.

\noindent
{\rm(iii)} \emph{Methodological role.}
Note that a single particle trajectory may intersect the sampling path at arbitrarily separated times. Consequently, even after averaging over the sampling walk, the sequence $\bigl\{\varphi(\xi(n,S_n))\bigr\}_{n\geq 1}$
is not finitely dependent. Hence there is no bounded-degree dependency graph. A truncation or mixing
argument may be possible under non-zero drift, but it would require
additional quantitative arguments beyond those developed here. The Poisson-chaos formulation is therefore useful not merely as a device for obtaining the stated rates. It retains the shared-particle dependence explicitly at the level of the underlying path-space Poisson process. The Malliavin--Stein approach avoids this auxiliary approximation and gives a unified and structurally explicit treatment of all three classes of observables.

\noindent
{\rm(iv)} \emph{Possible optimality of $Y_{N,\varphi}$.}
For the drifted polynomial functionals, we do not make a rigorous claim concerning the optimality of the rate $N^{-1/2}$ in Theorem~\ref{thm:Yk-main}. We only record the following heuristic observation.
Write
$\kappa_3(X)=\mathbb E\big[(X-\mathbb E X)^3\big].$
For a fixed $\varphi$, it follows from the finite
Poisson--Charlier expansion \eqref{eq:single-site-k-chaos}
and Lemma~\ref{lem:four-time-connected-Yk} that
\[
  \kappa_3(F_{N,\varphi})
  =
  \frac{\Gamma_\varphi}{\mathfrak c_\varphi^{3/2}}
  N^{-1/2}
  +o(N^{-1/2})
\]
for some finite constant $\Gamma_\varphi\in\mathbb R$.
When $\Gamma_\varphi\neq0$, by analogy with the classical
Berry--Esseen setting, this suggests that the rate $N^{-1/2}$ is
optimal in the non-cancellation case. Turning this observation into a matching lower bound in the
Wasserstein distance would require additional arguments beyond those developed here. We do not pursue this further.
When $\Gamma_\varphi=0$, such a cancellation may leave room for a faster rate
for special choices of $\varphi$, but does not by itself imply
an improved Wasserstein bound.

\noindent
{\bf General notation.}
For the reader’s convenience, we collect some notation used throughout the paper.
\begin{itemize}
\item $C,c$ denote generic positive constants whose values may change
from line to line. Unless otherwise indicated, they are independent
of $N$ and may depend only on the fixed parameters relevant to the
statement under consideration. In particular, they may depend on
$(\lambda,A)$ in the fixed-region setting and on
$(\varphi,\lambda,\mathfrak p)$ in the path-sampled setting.
The notation $C(\varphi,\lambda,\mathfrak p)$ allows dependence on
both the coefficients of $\varphi$ and its degree
$k=\deg(\varphi)$, the maximal chaos order. When the dependence on
$\varphi$ is only through $k$, we write $C(k,\lambda,\mathfrak p)$
to emphasize that the constant is independent of the coefficients
of $\varphi$. Other parameter dependences are indicated explicitly
when relevant, for example by $C(\lambda,A)$ or $C(q)$.

\item For quantities $a=a(x)$ and $b=b(x)\ge 0$ in the relevant limiting regime
(typically $x=N\to\infty$ or $x=t\to\infty$), we write $a\lesssim b$ or
$a=O(b)$ if $|a(x)|\le C\,b(x)$ for some constant $C>0$ and all sufficiently
large $x$, and $a\asymp b$ if both $a\lesssim b$ and $b\lesssim a$ hold.
We write $a=o(b)$ if $a(x)/b(x)\to0$ as $x\to\infty$.

\item We write $a_n\sim b_n$ if $a_n/b_n\to1$ in the regime under consideration.

\item Whenever a quantitative estimate is first proved only for all
sufficiently large $N$, the corresponding constant is adjusted, if
necessary, to cover the finitely many remaining values of $N$.
\end{itemize}

\noindent
{\bf Organization of the paper.} The rest of the paper is organized as follows.
In Section~\ref{sec:pre} we recall the Poisson Malliavin framework and the
Malliavin--Stein bound used in the proofs.
In Section~\ref{sec:3} we apply this bound to the fixed-region functionals
$W_N(A)$ and $D_N(A)$ and prove Theorem~\ref{thm:WN}.
Section~\ref{sec:Yk} establishes the annealed quantitative CLT for the path-sampled observables $Y_{N,\varphi}$ and
proves Theorem~\ref{thm:Yk-main}. Appendix~\ref{app:symmetric-sampling} discusses the symmetric sampling regime and contains the corresponding supplementary arguments.

\section{Preliminaries }\label{sec:pre}
\subsection{Malliavin calculus on the Poisson space }\label{malliavincalculus}\par\noindent

In this subsection, we briefly introduce some basic notions and properties of Malliavin calculus on the Poisson space. For more details on Malliavin calculus for Poisson functionals, we refer to \cite{nualart1990anticipative,last2016normal,peccati2010stein,last2016stochastic,last2018lectures} and the references therein.

\noindent
{\bf Poisson random measure.} Let $(\mathbb X,\mathscr{X})$ be a standard Borel space, which is equipped with a $\sigma$-finite measure $\mu$.
By $\eta$ we denote a Poisson random measure on $\mathbb X$ with control measure $\mu$, which is defined on an
underlying probability space $(\Omega,\mathcal{F},\mathbb{P})$.
By a Poisson random measure $\eta$ with control measure $\mu$, we mean the following:
$\eta=\{\eta(B): B\in \mathscr{X}_0\}$ is a collection of random variables indexed by the elements of
$\mathscr{X}_0=\{B\in\mathscr{X}:\mu(B)<\infty\}$ such that $\eta(B)$ is Poisson distributed with mean $\mu(B)$
for each $B\in\mathscr{X}_0$, and for all $n\in\mathbb{N}$, the random variables
$\eta(B_1),\ldots,\eta(B_n)$ are independent whenever $B_1,\ldots,B_n$ are disjoint sets from $\mathscr{X}_0$. Let ${\bf N}_\sigma= {\bf N}(\mathbb X)$ be the space of integer-valued $\sigma$-finite measures on $\mathbb X$ equipped with the smallest sigma-field ${\mathscr N}$ making the mappings $\nu \to \nu(B)$ measurable for all $B\in\mathscr{X}$. We remark that the Poisson random measure $\eta$ is a random element in the space ${\bf N}_\sigma$. The distribution of $\eta$ (on the space ${\bf N}_\sigma$) will be denoted by $\mathbb{P}_\eta$. For more details, see \cite[Chapter~VI]{ccinlar2011probability}
and \cite{last2018lectures}.

\noindent
\textbf{$L^1$- and $L^2$-spaces.}
For $n\in\mathbb N$, write $\mu^{\otimes n}$ for the $n$-fold
product measure of $\mu$, and denote by
$L^1(\mu^{\otimes n})$ and $L^2(\mu^{\otimes n})$
the spaces of integrable and square-integrable functions with respect
to $\mu^{\otimes n}$, respectively. We use the convention $L^2(\mu^{\otimes 0})=\mathbb R$.
The scalar product and the norm in $L^2(\mu^{\otimes n})$ are denoted by $\langle \cdot , \cdot \rangle_{L^2(\mu^{\otimes n})}$ and $\|\cdot\|_n$, respectively. 
Moreover, let us denote by $L^2(\mathbb{P}_\eta)$ the space of square-integrable functionals of a Poisson random measure $\eta$. 

\noindent
{\bf Chaos decomposition of $L^2$ functionals.}
For any $g\in L^1(\mu)$, the stochastic integral of $g$ with respect to the compensated Poisson measure $\hat{\eta}:= \eta - \mu$ is denoted by $I_1(g)=\int_\mathbb X g(\gamma)\hat{\eta}(\mathrm d \gamma)$. For $n\geq 2$ and $g_n\in L^1(\mu^{\otimes n})$, the multiple Wiener--It\^{o} integral of $g_n$ is given by 
\begin{align}\label{n:integral} 
I_n(g_n)=\sum_{J\subset[n]}(-1)^{n-|J|}\int_{\mathbb X^n}g_n(\gamma_1,\dots,\gamma_n)\eta^{(|J|)}(\mathrm d \gamma_{J})\mu^{\otimes(n-|J|)}(\mathrm d \gamma_{J^C}).
\end{align}
In particular, if $g_n=g^{\otimes n}$ with some $g\in L^1(\mu)\cap L^2(\mu)$, \eqref{n:integral} is reduced to 
\begin{align*}
I_n(g_n)=\sum_{k=0}^n(-1)^{n-k}\binom {n}{k}\eta^{(k)}(g^{\otimes k})[\mu(g)]^{n-k},
\end{align*}
where, for $k\geq 2$, $\eta^{(k)}$ denotes the $k$-th \emph{factorial measure} of $\eta$. For $B\in{\mathscr X}$, we have 
$$
\eta^{(k)}(B^k)=\eta(B)(\eta(B)-1)\cdots(\eta(B)-k+1).
$$

For every $f\in L^2(\mathbb P_\eta)$, there exists a unique sequence of symmetric functions $g_n\in L^2(\mu^{\otimes n}),n\ge0,$ such that
\begin{equation}\label{CD1}
  f(\eta)=\sum_{n=0}^\infty I_n(g_n),
\end{equation}
where the series converges in $L^2(\mathbb{P}_\eta)$. Moreover,
\begin{equation}\label{e:I_n}
  \mathbb E\big[I_n(g_n)I_m(g_m)\big]
  =\mathbf 1_{\{n=m\}}n!\,\langle g_n,g_m\rangle_{L^2(\mu^{\otimes n})}.
\end{equation}
The representation \eqref{CD1} is called the chaos expansion
of $f$ and we say that $f$ has a finite chaos expansion if only finitely many of
the functions $g_n$ are non-vanishing. In particular, \eqref{CD1} together with the orthogonality relation
\eqref{e:I_n} leads to the variance formula
\begin{equation}\label{chao-var}
\operatorname{Var}(f)=\sum_{n=1}^{\infty} n!\,\|g_n\|^{2}_n.
\end{equation}

\noindent
{\bf Malliavin operators.} We next recall the Malliavin operators used below. For a measurable function $f:{\bf N}_\sigma\to\mathbb R$,
the Malliavin derivative in the direction $\gamma\in\mathbb X$
is defined by
\begin{align*}
{D}_\gamma f(\eta)=f(\eta+\delta_\gamma)-f(\eta),\quad \eta\in{\bf N}_\sigma,
\end{align*}
where the Dirac measure $\delta_\gamma$ is defined by
\[
\delta_\gamma(B)=\mathbf{1}_B(\gamma),\quad B\in\mathscr{X}.
\]

Suppose that $f\in L^2(\mathbb P_\eta)$ admits the decomposition
\eqref{CD1}. We say that $f$ belongs to the domain of ${D}$, denoted by
$f\in {\rm dom}({D})$, if $\sum_{n=1}^\infty nn!\|g_n\|_n^2<\infty$. In this case, for $\mu$-a.e. $\gamma\in\mathbb{X}$, we have $\mathbb{P}_\eta$-a.s.
\begin{align}\label{De}
{D}_\gamma f=\sum_{n=1}^\infty nI_{n-1}(g_{n}(\gamma,\cdot)).
\end{align}
For $f\in L^2(\mathbb P_\eta)$ satisfying \eqref{CD1}, we say that
$f$ belongs to the domain of the Ornstein--Uhlenbeck operator ${L}$,
denoted by $f\in {\rm dom}({L})$, if $\{g_n\}_{n=0}^\infty$ satisfies $\sum_{n=1}^\infty n^2n!\|g_n\|^2_{n}<\infty$.
In this case, one has 
\begin{align*}
{L}f=-\sum_{n=1}^\infty nI_n(g_n)
\end{align*}
and its inverse is denoted by $L^{-1}$. In terms of the chaos expansion of a centered random
variable $f\in L^{2}(\mathbb P_\eta)$, i.e.\ $\mathbb E(f)=0$, it is given by
\begin{align}\label{L_inverse}
L^{-1}f=-\sum_{n=1}^{\infty}\frac{1}{n}\, I_n(g_n).
\end{align}

\noindent
\textbf{Contractions.} For integers $q_1, q_2 \geq 1$, let $r \in \{0,\ldots, \min(q_1,q_2)\}$, $\ell \in \{1,\ldots,r\}$, and let $f_1 \in L^2(\mu^{\otimes q_1})$ and $f_2 \in L^2(\mu^{\otimes q_2})$ be symmetric 
functions. 
The contraction kernel $f_1 \star_r^\ell f_2$ on $\mathbb X^{q_1+q_2-r-\ell}$ acts on the tensor product $f_1 \otimes f_2$ 
first by identifying $r$ variables and then integrating out $\ell$ among them. More formally,
\begin{align*}
& f_1 \star_r^\ell f_2(z_1,\ldots,z_{r-\ell}, t_1,\ldots,t_{q_1-r}, s_1,\ldots,s_{q_2-r}) \\
&= \int_{\mathbb X^\ell} f_1(\gamma_1,\ldots,\gamma_\ell, z_1,\ldots,z_{r-\ell}, t_1,\ldots,t_{q_1-r}) \\
&\quad \times f_2(\gamma_1,\ldots,\gamma_\ell,z_1,\ldots,z_{r-\ell}, s_1,\ldots,s_{q_2-r}) 
  \mu^{\otimes \ell}(\mathrm{d}(\gamma_1,\ldots,\gamma_\ell)).
\end{align*}

In addition, we define
\[
f_1 \star_r^0 f_2(\gamma_1,\ldots,\gamma_r,t_1,\ldots,t_{q_1-r}, s_1,\ldots,s_{q_2-r}) 
= f_1(\gamma_1,\ldots,\gamma_r,t_1,\ldots,t_{q_1-r}) f_2(\gamma_1,\ldots,\gamma_r,s_1,\ldots,s_{q_2-r}).
\]

Besides the contraction $f_1 \star_r^\ell f_2$, we will also deal with its canonical symmetrization
$f_1  \widetilde{\star}_r^\ell f_2$. It is defined by
\[
(f_1  \widetilde{\star}_r^\ell f_2)(x_1,\ldots,x_{q_1+q_2-r-\ell})
= \frac{1}{(q_1+q_2-r-\ell)!} \sum_{\pi} (f_1 \star_r^\ell f_2)(x_{\pi(1)},\ldots,x_{\pi(q_1+q_2-r-\ell)}),
\]
where the sum runs over all $(q_1+q_2-r-\ell)!$ permutations of $\{1,\ldots,q_1+q_2-r-\ell\}$.

\noindent
\textbf{Product formula.} Let $q_1,q_2 \ge 1$ and let $f_1 \in L^2(\mu^{\otimes q_1})$ and $f_2 \in L^2(\mu^{\otimes q_2})$
be symmetric functions. Assume that, for every
$0\le r\le \min(q_1,q_2)$ and $0\le \ell\le r$, the contraction
$f_1\star_r^\ell f_2$ is well defined and belongs to
$L^2(\mu^{\otimes q_1+q_2-r-\ell})$. Then
\begin{align}\label{mutiIn}
I_{q_1}(f_1) I_{q_2}(f_2)
= \sum_{r=0}^{\min(q_1,q_2)} r! \binom{q_1}{r} \binom{q_2}{r}
\sum_{\ell=0}^{r} \binom{r}{\ell}\,
I_{q_1+q_2-r-\ell}(f_1\widetilde{\star}_r^\ell f_2).
\end{align}
See \cite[Proposition 6.5.1]{peccati2011wiener}.

\subsection{Malliavin--Stein bound}\label{sec:MSB}\par\noindent

Besides Malliavin calculus, our proofs of Theorems~\ref{thm:WN} and~\ref{thm:Yk-main} rely on Stein's method that goes back to Stein \cite{stein1972bound,stein1986approximate} and is a powerful tool for proving limit theorems. For a detailed and more general introduction to this topic, we refer to \cite{MR2732624,chen2005stein,stein1986approximate}. 

For real-valued random variables $X,Y$ and a nonempty class $\mathcal H$
of measurable functions $h:\mathbb R\to\mathbb R$ such that
$\mathbb E[|h(X)|]+\mathbb E[|h(Y)|]<\infty$ for every $h\in\mathcal H$,
define
\[
  d_{\mathcal H}(X,Y)
  =
  \sup_{h\in\mathcal H}
  \bigl|\mathbb Eh(X)-\mathbb Eh(Y)\bigr|.
\]
In particular, for integrable real-valued $X,Y$, the Wasserstein distance
is obtained by taking all $1$-Lipschitz functions:
\begin{align*}
  d_W(X,Y)
  =
  \sup\bigl\{
    |\mathbb Eh(X)-\mathbb Eh(Y)|:
    \|h\|_{\rm Lip}\le1
  \bigr\},
\end{align*}
where
\[
  \|h\|_{\rm Lip}
  :=
  \sup_{x\ne y}
  \frac{|h(x)-h(y)|}{|x-y|}.
\]

The following result gives an upper bound on the Wasserstein distance $d _W(F,Z)$ in terms of Malliavin calculus; see \cite[Theorem 7]{bourguin2016malliavin} for details.
\begin{lemma}
With the notation and assumptions as above, let $F\in\rm dom(D)$ with $\mathbb{E}[F]=0$. Then we have
\begin{align}\label{SB2}
\begin{split}
d_{W}(F,Z)&\leq\sqrt{\frac{2}{\pi}}\mathbb E\left[|1-\langle{D}F,-{D}{L}^{-1}F\rangle_{L^2{(\mu)}}|\right]+\int_{\mathbb X}\mathbb E\left[|{D}_\gamma F|^2|{D}_\gamma{L}^{-1}F|\right]\mu{(\mathrm d\gamma)}\\
&\leq\sqrt{\frac{2}{\pi}\mathbb E\left[(1-\langle{D}F,-{D}{L}^{-1}F\rangle_{L^2{(\mu)}})^2\right]}+\int_{\mathbb X}\mathbb E\left[|{D}_\gamma F|^2|{D}_\gamma{L}^{-1}F|\right]\mu{(\mathrm d\gamma)},
\end{split}
\end{align}
where we use the standard notation
\[
\langle DF, -DL^{-1}F\rangle_{L^2{(\mu)}}
= - \int_{\mathbb X} (D_{z}F) (D_{z}L^{-1}F) \mu(\mathrm d z).
\]   
\end{lemma} 
\subsection{The occupation field as a functional of a Poisson random measure}\par\noindent

To prove Theorems~\ref{thm:WN} and~\ref{thm:Yk-main}, we need to apply the bound in
\eqref{SB2}. To this end, we realize the occupation field
$\xi=\{\xi(n,x): n\in\mathbb N_0,\ x\in\mathbb Z\}$ as a functional of a Poisson random measure on the path space.

We specialize the general framework from Subsection~\ref{malliavincalculus} to the measurable space
$(\mathbb X,\mathscr X)=(\mathbb W,\mathcal W)$, where
\[
\mathbb W
=
\bigl\{\, w=(w(n))_{n\in\mathbb N_0}:\; w(n)\in\mathbb Z,\ |w(n+1)-w(n)|=1\ \forall n\in\mathbb N_0 \,\bigr\},
\]
and $\mathcal W$ is the $\sigma$-field generated by the coordinate maps
\[
\pi_n:\mathbb W\to\mathbb Z,\qquad \pi_n(w):=w(n),\qquad n\in\mathbb N_0.
\]
For each $x\in\mathbb Z$, let
\[
\mathbb W_x=\{w\in\mathbb W:\ w(0)=x\}\in\mathcal W.
\]
Then $\{\mathbb W_x\}_{x\in\mathbb Z}$ is a measurable partition of $\mathbb W$.

Let $\mathbb P_x$ denote the law on $(\mathbb W,\mathcal W)$ of an SSRW on $\mathbb Z$ started from $x$. Equivalently, under $\mathbb P_x$, the coordinate
process $(\pi_n)_{n\in\mathbb N_0}$ is an SSRW with $\pi_0=x$.
Define a $\sigma$-finite measure $\mu$ on $(\mathbb W,\mathcal W)$ by
\[
\mu(\cdot)
=
\lambda\sum_{x\in\mathbb Z}\mathbb P_x(\cdot).
\]

Recall Definition~\ref{def:pfiw}. Using the family of particle trajectories
$\{X^{x,i}:x\in\mathbb Z,\ i\in\mathbb N\}$, define the random point measure
\[
\eta
=
\sum_{x\in\mathbb Z}\sum_{i=1}^{\xi(0,x)}\delta_{X^{x,i}} ,
\]
where $\delta_w$ stands for the Dirac measure concentrated at the point $w\in \mathbb W$. Then $\eta$ is a Poisson random measure on $(\mathbb W,\mathcal W)$ with control measure $\mu$.
Indeed, for each $x\in\mathbb Z$, the contribution
\[
\eta_x:=\sum_{i=1}^{\xi(0,x)}\delta_{X^{x,i}} 
\]
is a Poisson random measure on $(\mathbb W,\mathcal W)$ with intensity
$\lambda\,\mathbb P_x$, and the family $\{\eta_x\}_{x\in\mathbb Z}$ is independent; hence,
by superposition, $\eta=\sum_{x\in\mathbb Z}\eta_x$ is Poisson with intensity
$\mu=\lambda\sum_{x\in\mathbb Z}\mathbb P_x$.

For $(n,x)\in\mathbb N_0\times\mathbb Z$, define the cylinder set
\[
C_{n,x}
=
\{w\in\mathbb W:\ w(n)=x\}\in\mathcal W.
\]
Then the evaluation map $\nu\mapsto \nu(C_{n,x})$ is measurable on the space of counting
measures on $(\mathbb W,\mathcal W)$, and therefore
\begin{equation}\label{eq:xiFunctional}
\xi(n,x)
=
\eta(C_{n,x})
=
\int_{\mathbb W}\mathbf 1\{w(n)=x\}\,\eta(\mathrm dw)
=
\sum_{y\in\mathbb Z}\sum_{i=1}^{\xi(0,y)}\mathbf 1\{X_n^{y,i}=x\}.
\end{equation}
Thus $\xi(n,x)$ is a measurable functional of the Poisson random measure $\eta$.
More generally, for each $n\in\mathbb N_0$,
\begin{equation}\label{eq:xiMeasure}
\xi_n(B)
=
\int_{\mathbb W}\mathbf 1\{w(n)\in B\}\,\eta(\mathrm dw),
\qquad B\subset\mathbb Z,
\end{equation}
defines a random counting measure on $\mathbb Z$, and $\xi(n,x)=\xi_n(\{x\})$.

The representation \eqref{eq:xiFunctional} yields the standard
facts used later.

\begin{itemize}
\item \emph{One-time marginals.}
For every $(n,x)\in\mathbb N_0\times\mathbb Z$,
\[
\mu(C_{n,x})
=
\lambda\sum_{y\in\mathbb Z}\mathbb P_y(\pi_n=x)
=
\lambda\sum_{z\in\mathbb Z}\mathbb P_0(\pi_n=z)
=
\lambda.
\]
Hence
\[
\xi(n,x)\sim \mathrm{Poi}(\lambda),\qquad \mathbb E[\xi(n,x)]=\lambda,\qquad
\operatorname{Var}(\xi(n,x))=\lambda.
\]

\item \emph{Independence at a fixed time.}
For fixed $n$, the sets $\{C_{n,x}\}_{x\in\mathbb Z}$ are pairwise disjoint. Therefore
$\{\xi(n,x)\}_{x\in\mathbb Z}$ are independent, and by the previous item they are i.i.d.\
$\mathrm{Poi}(\lambda)$. Equivalently, $\xi_n$ in \eqref{eq:xiMeasure} is a Poisson random measure on
$\mathbb Z$ with intensity $\lambda\,\#(\cdot)$, where $\#(\cdot)$ denotes the counting measure.

\item \emph{Two-time covariance identity.}
For measurable sets $A_1,A_2\in\mathcal W$, a Poisson random measure satisfies
\[
\operatorname{Cov}\big(\eta(A_1),\eta(A_2)\big)=\mu(A_1\cap A_2).
\]
Applying this with $A_1=C_{n_1,x_1}$ and $A_2=C_{n_2,x_2}$, and using the Markov property and spatial
homogeneity of the SSRW, we obtain
\[
\operatorname{Cov}\big(\xi(n_1,x_1),\xi(n_2,x_2)\big)
=
\mu(C_{n_1,x_1}\cap C_{n_2,x_2})
=
\lambda\,Q_{|n_2-n_1|}(x_2-x_1).
\]
In particular, variables at different times are typically not independent.
\end{itemize}

The preceding representation immediately shows that the observables considered in this paper are
Poisson functionals.

First, for a finite nonempty set $A\subset\mathbb Z$ and
$n\in\mathbb N_0$, define
\[
  C_{n,A}
  =
  \{w\in\mathbb W:w(n)\in A\}
  =
  \bigcup_{x\in A}C_{n,x}.
\]
Set
\[
  f_N(w)
  =
  \sum_{n=1}^N\mathbf 1_{C_{n,A}}(w)
  =
  \sum_{n=1}^N\mathbf 1\{w(n)\in A\},
  \qquad w\in\mathbb W.
\]
Then
\[
  W_N(A)
  =
  \sum_{n=1}^N\sum_{x\in A}\xi(n,x)
  =
  \int_{\mathbb W}f_N(w)\,\eta(\mathrm dw).
\]
Since $0\le f_N\le N$ and
$\int_{\mathbb W}f_N(w)\,\mu(\mathrm dw)=N\lambda|A|<\infty$,
we have $f_N\in L^r(\mu)$ for every finite $r\ge1$. In particular,
\begin{equation}\label{eq:WN-proof-chaos}
  W_N(A)-\mathbb E[W_N(A)]
  =
  \int_{\mathbb W}f_N(w)\,\hat\eta(\mathrm dw)
  =
  I_1(f_N).
\end{equation}

Next, define the hitting set and its indicator kernel by
\[
H_{N,A}
=
\{w\in\mathbb W:\ \exists\,m\in\{1,\dots,N\}\text{ such that }w(m)\in A\},\qquad
  g_N=\mathbf 1_{H_{N,A}}.
\]
Then
\[
D_N(A)=\eta(H_{N,A}).
\]
Again, since $A$ is finite and $N<\infty$, we have $\mu(H_{N,A})<\infty$. So
$g_N\in L^r(\mu)$ for every finite $r\ge1$ and hence
\begin{equation}\label{eq:DN-proof-chaos}
  D_N(A)-\mathbb E[D_N(A)]
  =\int_{\mathbb W}g_N(w)\,\hat\eta(\mathrm dw)
  =I_1(g_N).
\end{equation}

Finally, consider the path-sampled polynomial functional \eqref{eq:Yk-def-intro}. Since the sampling
walk $S$ is independent of $\eta$, for every fixed realization
$s=(s_n)_{n\in\mathbb N_0}$ of the sampling path, the map
\[
\eta\longmapsto
Y_{N,\varphi}(\eta,s)
=
\sum_{n=1}^N\varphi\bigl(\eta(C_{n,s_n})\bigr)
\]
is a measurable functional of the Poisson random measure $\eta$. Equivalently,
\[
Y_{N,\varphi}
=
\sum_{n=1}^N\varphi(\xi(n,S_n))
=
\sum_{n=1}^N\varphi\bigl(\eta(C_{n,S_n})\bigr)
\]
is a measurable function of the pair $(\eta,S)$, and conditionally on $S$ it is a Poisson functional
of $\eta$. This conditional viewpoint will be used in Section~\ref{sec:Yk}.

\section{Proof of Theorem \ref{thm:WN}}\label{sec:3}

Throughout this section, $A\subset\mathbb Z$ is fixed, finite and
nonempty. We begin with Wasserstein bounds for normalized first-order
Poisson integrals.

\begin{lemma}[First-chaos Wasserstein bounds]
\label{lem:first-chaos-lower}
Let $(h_N)_{N\ge1}\subset L^1(\mu)\cap L^4(\mu)$ be nonnegative and
satisfy $\int_{\mathbb W}h_N(w)^2\,\mu(\mathrm dw)=1.$
Set $b_N =\int_{\mathbb W}h_N(w)^3\,\mu(\mathrm dw),     \quad r_N =\int_{\mathbb W}h_N(w)^4\,\mu(\mathrm dw).$
Then, for every $N\ge1$,
\begin{equation}\label{eq:first-chaos-upper}
  d_W\bigl(I_1(h_N),Z\bigr)\le b_N.
\end{equation}
If $b_N\to0$ and $r_N=o(b_N)$, then, for all sufficiently large $N$,
\begin{equation}\label{eq:first-chaos-lower}
  d_W\bigl(I_1(h_N),Z\bigr)
  \ge\frac{\mathrm{e}^{-1/2}}{12}\,b_N.
\end{equation}
\end{lemma}

\begin{proof}
By the Cauchy--Schwarz inequality, $b_N\le\sqrt{r_N}<\infty$.
Moreover, $b_N>0$ because $h_N\ge0$ and
$\int_{\mathbb W}h_N(w)^2\,\mu(\mathrm dw)=1$. Note that
$I_1(h_N)$ is centered and belongs to
$\operatorname{dom}(D)$. By \eqref{De} and \eqref{L_inverse},
\[
  D_wI_1(h_N)=h_N(w),
  \quad
  L^{-1}I_1(h_N)=-I_1(h_N).
\]
Thus the first term in \eqref{SB2} vanishes by normalization, while
the second term equals
\[
  \int_{\mathbb W}h_N(w)^3\,\mu(\mathrm dw)=b_N.
\]
This proves \eqref{eq:first-chaos-upper}.

For the lower bound, the exponential formula for compensated Poisson integrals and Taylor's formula give
\begin{align*}
  \mathbb E\big[\mathrm{e}^{\mathrm i I_1(h_N)}\big]
  &=
  \exp\left\{
    \int_{\mathbb W}
    \big(
      \mathrm{e}^{\mathrm i h_N(w)}-1-\mathrm i h_N(w)
    \big)\,
    \mu(\mathrm dw)
  \right\} \\
  &=
  \exp\left\{
    -\frac12-\mathrm i\frac{b_N}{6}+\mathcal E_N
  \right\},
  \qquad
  |\mathcal E_N|\le\frac{r_N}{24}.
\end{align*}
Here we used
\[
  \left|
    \mathrm{e}^{\mathrm i x}-1-\mathrm i x
    +\frac{x^2}{2}
    +\mathrm i\frac{x^3}{6}
  \right|
  \le\frac{|x|^4}{24},
  \qquad x\in\mathbb R.
\]
Since $b_N\to0$ and $r_N=o(b_N)$, expanding the exponential and
taking imaginary parts gives
\[
  \mathbb E\big[\sin I_1(h_N)\big]
  =
  -\frac{\mathrm{e}^{-1/2}}6\,b_N
  +O(r_N+b_N^2)
  =
  -\frac{\mathrm{e}^{-1/2}}6\,b_N+o(b_N).
\]
Finally, $x\mapsto\sin x$ is $1$-Lipschitz and
$\mathbb E[\sin Z]=0$. Hence
\[
  d_W\bigl(I_1(h_N),Z\bigr)
  \ge
  \left|
    \mathbb E\big[\sin I_1(h_N)\big]
  \right|.
\]
This proves
\eqref{eq:first-chaos-lower}.
\end{proof}

\subsection{Proof of Theorem \ref{thm:WN}~(i): the functional \texorpdfstring{$W_N(A)$}{W\_N(A)}}
\label{subsec:proof-WN}

Recall that $W_N(A)-\mathbb E[W_N(A)]=I_1(f_N).$
Hence, by the variance formula \eqref{chao-var},
\begin{equation}\label{eq:WN-proof-var0}
  \sigma_{N,W}^2:=\operatorname{Var}(W_N(A))
  =\|f_N\|_{L^2(\mu)}^2
  =\sum_{n,m=1}^N\mu(C_{n,A}\cap C_{m,A}).
\end{equation}
For $n\le m$, using the Markov property of the SSRW and
$\mu=\lambda\sum_{u\in\mathbb Z}\mathbb P_u$, we have

\begin{equation}\label{eq:muCnCm-rigorous}
\begin{aligned}
  \mu\big(C_{n,A}\cap C_{m,A}\big)
  &=\lambda\sum_{u\in\mathbb Z}\mathbb P_u\big(w(n)\in A,\ w(m)\in A\big) \\
  &=\lambda\sum_{u\in\mathbb Z}\sum_{x\in A}\mathbb P_u\big(w(n)=x\big)\mathbb P_x\big(w(m-n)\in A\big) \\
  &=\lambda\sum_{x\in A}\mathbb P_x\big(w(m-n)\in A\big)
   =\lambda\sum_{x,y\in A}Q_{m-n}(y-x).
\end{aligned}
\end{equation}
Substituting \eqref{eq:muCnCm-rigorous} into \eqref{eq:WN-proof-var0} and separating the diagonal terms
yields
\begin{equation}\label{eq:Var-expansion-rigorous}
  \operatorname{Var}\!\big(W_N(A)\big)
  =\lambda\Bigg[ N|A| + 2\sum_{n=1}^{N-1}(N-n)\sum_{x,y\in A}Q_n(y-x)\Bigg].
\end{equation}

We will use the following parity-corrected local central limit theorem to obtain
sharp asymptotics for \eqref{eq:Var-expansion-rigorous}.

\begin{lemma}\label{lem:heat-kernel-rigorous}
Recall that $Q_n(x)$ is the transition kernel of the one-dimensional
SSRW. Then
\[
  \sup_{x\in\mathbb Z}
  \left|
    \frac{\sqrt n}{2}\,Q_n(x)
    -\mathbf 1_{\{x\equiv n\!\!\!\!\pmod 2\}}
     \frac{1}{\sqrt{2\pi}}\mathrm{e}^{-x^2/(2n)}
  \right|\longrightarrow0
  \qquad\text{as }n\to\infty.
\]
Moreover, there exist absolute constants $C_0,c_0\in(0,\infty)$ such
that
\[
  Q_n(x)\le C_0n^{-1/2}\mathrm{e}^{-x^2/(c_0n)}
  \le C_0n^{-1/2},
  \qquad x\in\mathbb Z,\ n\ge1.
\]
\end{lemma}

\begin{proof}
The local-limit statement is the classical local central limit theorem for lattice
distributions; see, e.g., \cite[Chapter~VII, Theorem~1]{MR388499}. The Gaussian upper bound follows from \cite[Theorem~2.1]{MR1217561}.
\end{proof}

We can now refine \eqref{eq:Var-expansion-rigorous} as follows.
\begin{lemma}\label{lem:WN-var-sharp}
As $N\to\infty$,
\begin{equation}\label{eq:WN-var-sharp}
  \sigma_{N,W}^2
  =\frac{8\lambda|A|^2}{3\sqrt{2\pi}}\,N^{3/2}
   +o(N^{3/2}).
\end{equation}
In particular, $\sigma_{N,W}^2\asymp N^{3/2}$.
\end{lemma}

\begin{proof}
Since $A-A:=\{y-x:x,y\in A\}$ is finite,
Lemma~\ref{lem:heat-kernel-rigorous} implies, uniformly in $d\in A-A$,
\[
  Q_n(d)
  =\mathbf 1_{\{n\equiv d\!\!\!\!\pmod 2\}}
    \frac{2}{\sqrt{2\pi n}}+o(n^{-1/2}).
\]
For each $h\in\{0,1\}$, comparison with the corresponding integrals
gives
\[
  \sum_{\substack{1\le n<N\\ n\equiv h\!\!\!\!\pmod 2}}
  (N-n)n^{-1/2}
  =\frac23N^{3/2}+o(N^{3/2}).
\]
Consequently,
\begin{equation}\label{eq:sum-sigma}
  \sum_{n=1}^{N-1}(N-n)Q_n(d)
  =\frac{4}{3\sqrt{2\pi}}\,N^{3/2}+o(N^{3/2}),
\end{equation}
uniformly in $d\in A-A$. To control the accumulated local-limit error,
split the sum at a fixed cutoff: the initial segment contributes
$O(N)$, while a sufficiently late tail is bounded by an arbitrarily
small multiple of $N^{3/2}$.
Substituting \eqref{eq:sum-sigma} into
\eqref{eq:Var-expansion-rigorous} proves \eqref{eq:WN-var-sharp}.
\end{proof}

We next collect the kernel estimates needed for both Wasserstein
bounds.

\begin{lemma}\label{lem:WN-kernel-moments}
As $N\to\infty$,
\[
  \int_{\mathbb W}f_N(w)^3\,\mu(\mathrm dw)\asymp N^2,
  \qquad
  \int_{\mathbb W}f_N(w)^4\,\mu(\mathrm dw)\lesssim N^{5/2}.
\]
\end{lemma}

\begin{proof}
Lemma~\ref{lem:heat-kernel-rigorous} gives
\[
  Q_u(x)\lesssim (1+u)^{-1/2},
  \qquad u\in\mathbb N_0,\ x\in\mathbb Z.
\]
For $r\in\{3,4\}$ and
$1\le t_1\le\cdots\le t_r\le N$, the Markov property and spatial
homogeneity yield
\begin{align*}
  \mu\left(\bigcap_{a=1}^r C_{t_a,A}\right)
  &=
  \lambda
  \sum_{x_1,\dots,x_r\in A}
  \prod_{a=1}^{r-1}
  Q_{t_{a+1}-t_a}(x_{a+1}-x_a) \\
  &\lesssim \prod_{a=1}^{r-1}
  (1+t_{a+1}-t_a)^{-1/2}.
\end{align*}
Expanding $f_N^r$ and
ordering the time indices, with at most $r!$ permutations for each
ordered tuple, gives
\begin{align*}
  \int_{\mathbb W}f_N(w)^r\,\mu(\mathrm dw)
  &\le
  r!\sum_{1\le t_1\le\cdots\le t_r\le N}
  \mu\left(\bigcap_{a=1}^r C_{t_a,A}\right) \\
  &\lesssim N
  \left(
    \sum_{u=0}^{N-1}(1+u)^{-1/2}
  \right)^{r-1}
  \lesssim N^{(r+1)/2}.
\end{align*}

For the lower bound on
$\int_{\mathbb W}f_N(w)^3\,\mu(\mathrm dw)$, the
Cauchy--Schwarz inequality yields
\[
  \left(
    \int_{\mathbb W}f_N(w)^2\,\mu(\mathrm dw)
  \right)^2
  \le
  \left(
    \int_{\mathbb W}f_N(w)^3\,\mu(\mathrm dw)
  \right)
  \left(
    \int_{\mathbb W}f_N(w)\,\mu(\mathrm dw)
  \right).
\]
Therefore,
\[
  \int_{\mathbb W}f_N(w)^3\,\mu(\mathrm dw)
  \ge
  \frac{\sigma_{N,W}^4}{N\lambda|A|}
  \gtrsim N^2,
\]
where the last step follows from Lemma~\ref{lem:WN-var-sharp}.
\end{proof}

We now apply Lemma~\ref{lem:first-chaos-lower} with $h_N=\frac{f_N}{\sigma_{N,W}}.$
By Lemmas~\ref{lem:WN-var-sharp} and~\ref{lem:WN-kernel-moments},
the assumptions of Lemma~\ref{lem:first-chaos-lower} are satisfied. Hence
\[
  d_W\bigl(F_N^{(A)},Z\bigr)\asymp N^{-1/4}.
\]
This completes the proof of Theorem~\ref{thm:WN}~(i).

\subsection{Proof of Theorem \ref{thm:WN}~(ii): the functional \texorpdfstring{$D_N(A)$}{D\_N(A)}}
\label{subsec:proof-DN}

Recall that $D_N(A)-\mathbb E[D_N(A)]=I_1(g_N).$
Hence
\[
  \sigma_{N,D}^2:=\operatorname{Var}(D_N(A))
  =\|g_N\|_{L^2(\mu)}^2=\mu(H_{N,A}).
\]

To determine this variance, let $\widetilde X_m =\pi_m$,
$m\in\mathbb N_0$, where $\pi_m(w)=w(m)$ is the coordinate map on
$(\mathbb W,\mathcal W)$. Write $\mathbb E_0$ for expectation under
$\mathbb P_0$ and set
\[
  \tau_A^+ =\inf\{m\ge1:\widetilde X_m\in A\}.
\]
By the definition of $\mu=\lambda\sum_{x\in\mathbb Z}\mathbb P_x$, we have
\[
  \mu(H_{N,A})
  =\lambda\sum_{x\in\mathbb Z}\mathbb P_x(\tau_A^+\le N).
\]
Introduce the ranges
\[
  \mathcal R_N^+=\{\widetilde X_1,\dots,\widetilde X_N\},
  \qquad
  \mathcal R_N=\{\widetilde X_0,\dots,\widetilde X_N\},
  \qquad
  R_N=|\mathcal R_N|.
\]
Translation invariance gives
\[
  \mathbb P_x(\tau_A^+\le N)
  =\mathbb P_0\bigl(x\in A-\mathcal R_N^+\bigr),
\]
where $A-B:=\{a-b:a\in A,\ b\in B\}$.
Thus, by Tonelli's theorem,
\[
  \mu(H_{N,A})=\lambda\mathbb E_0\big[|A-\mathcal R_N^+|\big].
\]
Because the walk is nearest-neighbour, $\mathcal R_N^+$ is an integer
interval and
\[
  |\mathcal R_N^+|
  \le |A-\mathcal R_N^+|
  \le |\mathcal R_N^+|+\max A-\min A,
  \qquad
  \bigl||\mathcal R_N^+|-R_N\bigr|\le1.
\]
Consequently,
\begin{equation}\label{eq:DN-proof-range-RN}
  \mu(H_{N,A})=\lambda\mathbb E_0[R_N]+O(1).
\end{equation}

Let $M_N =\max_{0\le m\le N}\widetilde X_m$.
By the discrete reflection principle, for every integer $a\ge1$,
\[
  \mathbb P_0(M_N\ge a)
  =\mathbb P_0(\widetilde X_N\ge a)
   +\mathbb P_0(\widetilde X_N\ge a+1).
\]
Summing over $a$ and using symmetry and the nearest-neighbour
property gives
\[
  \mathbb E_0[R_N]
  =2\mathbb E_0[M_N]+1
  =2\mathbb E_0\big[|\widetilde X_N|\big]
   +\mathbb P_0(\widetilde X_N=0).
\]
Since $\mathbb E_0[\widetilde X_N^2]=N$, the sequence
$(|\widetilde X_N|/\sqrt N)_{N\ge1}$ is uniformly integrable.
The central limit theorem therefore gives
\[
  \frac{\mathbb E_0[R_N]}{\sqrt N}
  \longrightarrow2\mathbb E[|Z|]=2\sqrt{\frac2\pi}.
\]
Together with \eqref{eq:DN-proof-range-RN}, this proves
\begin{equation}\label{eq:DN-proof-sharp}
  \sigma_{N,D}^2\sim2\lambda\sqrt{\frac2\pi}\,N^{1/2},
\end{equation}
and in particular $\sigma_{N,D}\asymp N^{1/4}$.

Applying Lemma~\ref{lem:first-chaos-lower} with $h_N=g_N/\sigma_{N,D}$
and using \eqref{eq:DN-proof-sharp} yields
\[
  d_W\bigl(G_N^{(A)},Z\bigr)\asymp N^{-1/4}.
\]
This completes the proof of Theorem~\ref{thm:WN}~(ii).

\section{Proof of Theorem \ref{thm:Yk-main} }\label{sec:Yk}

The proof follows the standard Malliavin--Stein route:
first derive a conditional Poisson-chaos expansion for $Y_{N,\varphi}$
and show that $\sigma_{N,\varphi}^2\asymp N$, then apply an annealed Malliavin--Stein bound and estimate the two
resulting terms.
\subsection{Auxiliary estimates}\label{sec:auxiliary}

We collect a few drift-induced estimates that will be used repeatedly below. Recall that $S=(S_n,n\in\mathbb N_0)$ is an independent nearest-neighbour walk on $\mathbb Z$ with i.i.d.\ increments
$\Delta_n=S_n-S_{n-1}\in\{-1,+1\}$ satisfying
\begin{equation}\label{eq:drift-assumption-Yk}
  \mathbb P_S(\Delta_n=+1)=\mathfrak p\neq\tfrac12,\qquad
  \mathbb P_S(\Delta_n=-1)=1-\mathfrak p,\qquad
  v=\mathbb E_S[\Delta_1]=2\mathfrak p-1\neq0.
\end{equation}

\begin{lemma}\label{lem:ballistic-Y2}
Under \eqref{eq:drift-assumption-Yk}, there exists $c:=c(\mathfrak p)\in (0,\infty)$ such that for all $n\ge1$,
\[
  \mathbb P_S\Big(\big|S_n-v n\big|\ge \tfrac{|v|}{2}n\Big)\le 2 \mathrm{e}^{-c n}.
\]
In particular, $\mathbb P_S\Big(|S_n|\le \tfrac{|v|}{2}n\Big)\le 2 \mathrm{e}^{-c n}$.

\end{lemma}

\begin{proof}
Apply Hoeffding's inequality to $\sum_{i=1}^n(\Delta_i-v)$ to obtain the first bound.
For the second, note that on $\{|S_n|\le \frac{|v|}{2}n\}$,
\[
  |S_n-v n|\ge |v|n-|S_n|\ge \frac{|v|}{2}n,
\]
so $\{|S_n|\le \frac{|v|}{2}n\}\subseteq \{|S_n-v n|\ge \frac{|v|}{2}n\}$.
\end{proof}

\begin{lemma}[Exponential moment decay under drift]\label{lem:moments-decay-Yk}
Fix $m\in\mathbb N$. Under \eqref{eq:drift-assumption-Yk}, there exist constants $c:=c(\mathfrak p),C:=C(m)\in(0,\infty)$ such
that for all $r\ge1$,
\[
  \mathbb E_S\big[Q_r(S_r)^m\big]\le C\,r^{-m/2}\,\mathrm{e}^{-cr}.
\]
In particular, for every fixed $m$ the sequence $r\mapsto \mathbb E_S[Q_r(S_r)^m]$ is summable and
$r\mapsto r\,\mathbb E_S[Q_r(S_r)^m]$ is also summable.
\end{lemma}

\begin{proof}
Let $\theta=|v|/2>0$ and $E_r=\{|S_r|\le\theta r\}$.
We split the expectation according to $E_r$.
On $E_r$, $\sup_x Q_r(x)\le C_0 r^{-1/2}$ and by Lemma~\ref{lem:ballistic-Y2}, we have 
\[
  \mathbb E_S\big[Q_r(S_r)^m\mathbf 1_{E_r}\big]
  \le C r^{-m/2}\mathbb P(E_r)
  \le C r^{-m/2} \mathrm{e}^{-cr}.
\] On $E_r^c$ we use Lemma~\ref{lem:heat-kernel-rigorous}:
\[
  Q_r(S_r)^m
  \le C r^{-m/2}\exp\!\Big(-\frac{m S_r^2}{C r}\Big)
  \le C r^{-m/2}\exp\!\Big(-\frac{m\theta^2}{C}\,r\Big)
  \le C r^{-m/2}\mathrm{e}^{-cr},
\]
and hence $\mathbb E_S[Q_r(S_r)^m\mathbf 1_{E_r^c}]\le C r^{-m/2}\mathrm{e}^{-cr}$.
\end{proof}

We also record a multi-convolution counting bound used repeatedly.

\begin{lemma}\label{lem:conv-count-Yk}
Let $M\in\mathbb N$ and let $(u^{(j)}_n)_{n\ge1}$, $j=1,\dots,M$, be nonnegative sequences with
$\sum_{n\ge1}u^{(j)}_n<\infty$ for every $j$. Then for all $N\ge1$,
\[
  \sum_{n_1,\dots,n_M\ge1}(N-n_1-\cdots-n_M)_+\prod_{j=1}^M u^{(j)}_{n_j}
  \;\le\;
  N\prod_{j=1}^M\sum_{n\ge1}u^{(j)}_n.
\]
\end{lemma}

\begin{proof}
Use the fact $(N-n_1-\cdots-n_M)_+\le N$ and Tonelli's theorem.
\end{proof}

\subsection{Conditional Poisson-chaos expansion}

In this subsection we derive a conditional Poisson-chaos expansion for $Y_{N,\varphi}$. 
For a fixed realization of $S$, write
\[
 \mathcal A_n=C_{n,S_n},\qquad n=1,\dots,N,
\]
where the cylinder set $C_{n,x}=\{w\in\mathbb W:w(n)=x\}$. For later use, note that for all $1\le n,m\le N$,
\[
  \mu(\mathcal A_n)=\lambda,
  \qquad
  \mu(\mathcal A_n\cap \mathcal A_m)=\lambda Q_{|m-n|}(S_m-S_n).
\]
We first compute the conditional mean.

\begin{lemma}\label{lem:Yk-mean}
For every $N\ge1$,
\[
  \mathbb E_\xi[Y_{N,\varphi}\mid S]
  =N\,\mathbb E\big[\varphi(\mathrm{Poi}(\lambda))\big]
  \qquad \mathbb P_S\text{-a.s.}.
\]
In particular, $\mathbb E[Y_{N,\varphi}]
  =N\,\mathbb E\big[\varphi(\mathrm{Poi}(\lambda))\big]$.

\end{lemma}

\begin{proof}
Conditionally on $S$, for each $n=1,\dots,N$,
\[
  \mathbb E_\xi[\varphi(\xi(n,S_n))\mid S]
  =\sum_{x\in\mathbb Z}\mathbf 1_{\{S_n=x\}}\,\mathbb E_\xi[\varphi(\xi(n,x))]
  =\mathbb E[\varphi(\mathrm{Poi}(\lambda))],
\]
since $\xi(n,x)\sim\mathrm{Poi}(\lambda)$ for every fixed $(n,x)$.
Summing over $n$ yields
\[
  \mathbb E_\xi[Y_{N,\varphi}\mid S]
  =N\,\mathbb E[\varphi(\mathrm{Poi}(\lambda))].
\]
Taking expectation with respect to $S$ gives the unconditional identity.
\end{proof}

To expand single-site observables in Poisson chaoses, we use the orthogonal Charlier basis on
$\mathrm{Poi}(\lambda)$ together with an It\^o-type identity for indicator kernels.

Let \(C_n^{\mathrm{std}}(x;\lambda)\) denote the standard Charlier polynomial
as in \cite{MR2656096}, and define the corresponding monic
Charlier polynomial by
\[
  C_n(x;\lambda)=(-\lambda)^n C_n^{\mathrm{std}}(x;\lambda),
  \qquad n\ge 0.
\]
Then
\[
  C_0(x;\lambda)\equiv 1,
  \qquad
  C_1(x;\lambda)=x-\lambda.
\]
Moreover, by rescaling the standard formulas
\cite[Equations (9.14.1)--(9.14.3)]{MR2656096},
the monic Charlier polynomials satisfy the three-term recurrence
\begin{equation}\label{eq:Charlier-recursion}
  (x-\lambda)\,C_r(x;\lambda)
  = C_{r+1}(x;\lambda)+r\,C_r(x;\lambda)+r\lambda\,C_{r-1}(x;\lambda),
  \qquad r\ge 1.
\end{equation}
Let $\mathcal A\in\mathcal W$ satisfy $\mu(\mathcal A)=\lambda$ and set $N_\mathcal{A}=\eta(\mathcal A)\sim\mathrm{Poi}(\lambda)$. Then the orthogonality relation for the
standard Charlier polynomials, rewritten in terms of Poisson expectation, yields \(\mathbb E[C_r(N_{\mathcal A};\lambda)]=0\)
for \(r\ge1\), and
\begin{equation}\label{eq:Charlier-orth2}
  \mathbb E\!\left[C_r(N_\mathcal{A};\lambda)\,C_s(N_\mathcal{A};\lambda)\right]
  = \mathbf 1_{\{r=s\}}\, r!\,\lambda^r,
  \qquad r,s\ge 0.
\end{equation}

\begin{lemma}[Charlier--It\^o identity for indicator kernels]\label{lem:Charlier-Ito-indicator}
For every $r\ge1$,
\begin{equation}\label{eq:Charlier-Ito}
  I_r\big(\mathbf 1_\mathcal{A}^{\otimes r}\big)=C_r(N_\mathcal{A};\lambda).
\end{equation}
In particular,
\[
  I_1(\mathbf 1_\mathcal{A})=N_\mathcal{A}-\lambda=C_1(N_\mathcal{A};\lambda),
  \qquad
  I_2(\mathbf 1_\mathcal{A}^{\otimes2})=N_\mathcal{A}^2-(2\lambda+1)N_\mathcal{A}+\lambda^2=C_2(N_\mathcal{A};\lambda).
\]
\end{lemma}

\begin{proof}
Set $\mathcal P_r=I_r(\mathbf 1_\mathcal{A}^{\otimes r})$ for $r\ge0$ with the convention $\mathcal P_0=I_0(1)=1$.
Applying the Poisson product formula \eqref{mutiIn} with $q_1=1$, $q_2=r$, $f_1=\mathbf 1_\mathcal{A}$ and
$f_2=\mathbf 1_\mathcal{A}^{\otimes r}$ yields, for all $r\ge1$,
\[
  I_1(\mathbf 1_\mathcal{A})\,I_r(\mathbf 1_\mathcal{A}^{\otimes r})
  =I_{r+1}(\mathbf 1_\mathcal{A}^{\otimes(r+1)})+r\,I_r(\mathbf 1_\mathcal{A}^{\otimes r})+r\lambda\,I_{r-1}(\mathbf 1_\mathcal{A}^{\otimes(r-1)}).
\]
Since $I_1(\mathbf 1_\mathcal{A})=N_\mathcal{A}-\lambda$, this can be rewritten as
\[
  (N_\mathcal{A}-\lambda)\,\mathcal P_r
  =
  \mathcal P_{r+1}+r\,\mathcal P_r+r\lambda\,\mathcal P_{r-1},
\]
with $\mathcal P_0=1$ and $\mathcal P_1=N_\mathcal{A}-\lambda$.
This coincides with the recurrence \eqref{eq:Charlier-recursion}
and the initial conditions of the monic Charlier polynomials.
Hence
\[
  \mathcal P_r=C_r(N_\mathcal{A};\lambda)\qquad\text{for all }r\ge1,
\]
which proves \eqref{eq:Charlier-Ito}.
\end{proof}

We now expand the centered single-site polynomial observable
\[
  \varphi(\xi(n,x))-\mathbb E[\varphi(\xi(n,x))].
\]

\begin{lemma}[Single-site polynomial Poisson--Charlier expansion and coefficients]\label{lem:single-site-k-chaos}
There exist unique coefficients
$c_{\varphi,1},\dots,c_{\varphi,k}$ such that
\begin{equation}\label{eq:single-site-k-chaos}
  \varphi(N_\mathcal{A})-\mathbb E[\varphi(N_\mathcal{A})]
  =\sum_{q=1}^k c_{\varphi,q}\,I_q\big(\mathbf 1_\mathcal{A}^{\otimes q}\big).
\end{equation}
Moreover, these coefficients are given by
\begin{equation}\label{eq:ckq-Charlier-projection}
  c_{\varphi,q}
  =\frac{\mathbb E\!\big[\varphi(N_\mathcal{A})\,C_q(N_\mathcal{A};\lambda)\big]}{q!\,\lambda^q},
  \qquad 1\le q\le k.
\end{equation}
In particular, $c_{\varphi,k}=\beta_k\neq 0$.

\end{lemma}

\begin{proof}
Since $\{C_q(\cdot;\lambda)\}_{q\ge0}$ form a complete orthogonal system in
$L^2(\mathrm{Poi}(\lambda))$ and $\varphi$ is a polynomial of degree $k$, there exist unique
coefficients $d_{\varphi,q}$ such that
\[
  \varphi(N_\mathcal{A})-\mathbb E[\varphi(N_\mathcal{A})]
  =
  \sum_{q=1}^k d_{\varphi,q}\,C_q(N_\mathcal{A};\lambda).
\]
Taking the inner product with $C_q(N_\mathcal{A};\lambda)$ and using \eqref{eq:Charlier-orth2} gives
\[
  d_{\varphi,q}
  =
  \frac{\mathbb E[\varphi(N_\mathcal{A})\,C_q(N_\mathcal{A};\lambda)]}{\mathbb E[C_q(N_\mathcal{A};\lambda)^2]}
  =
  \frac{\mathbb E[\varphi(N_\mathcal{A})\,C_q(N_\mathcal{A};\lambda)]}{q!\lambda^q}.
\]
Setting $c_{\varphi,q}=d_{\varphi,q}$ yields \eqref{eq:ckq-Charlier-projection}. By
Lemma~\ref{lem:Charlier-Ito-indicator},
\[
  C_q(N_\mathcal{A};\lambda)=I_q(\mathbf 1_\mathcal{A}^{\otimes q}),
\]
so the above expansion translates into \eqref{eq:single-site-k-chaos}. Finally, since $\varphi$ has
degree $k$ with leading coefficient $\beta_k$ and $C_k(\cdot;\lambda)$ is monic of degree $k$, we have
$c_{\varphi,k}=\beta_k\neq0$.
\end{proof}

We can now sum the single-site expansion along the sampled path to obtain a conditional
Poisson-chaos expansion for $Y_{N,\varphi}$. For fixed $S$ and $q=1,\dots,k$, define
\begin{equation}\label{eq:FqS-def}
  f_{q,S}(w_1,\dots,w_q)=\sum_{n=1}^N \prod_{j=1}^q \mathbf 1_{\mathcal{A}_n}(w_j),
  \qquad (w_1,\dots,w_q)\in\mathbb W^q.
\end{equation}
Each $f_{q,S}$ is symmetric. Moreover,
\[
  \|f_{q,S}\|_{L^2(\mu^{\otimes q})}^2
  =\sum_{n,m=1}^N \mu(\mathcal{A}_n\cap \mathcal{A}_m)^q
  \le N^2\lambda^q<\infty,
\]
so $f_{q,S}\in L_s^2(\mu^{\otimes q})$.
Since
\[
  \mathbf 1_{\mathcal A_n}(w)
  =\sum_{x\in\mathbb Z}
    \mathbf 1_{\{S_n=x\}}\mathbf 1_{C_{n,x}}(w),
\]
the map
\[
  (S,w_1,\ldots,w_q)\longmapsto f_{q,S}(w_1,\ldots,w_q)
\]
is jointly measurable.

Conditionally on $S$, applying Lemma~\ref{lem:single-site-k-chaos} with $\mathcal{A}=\mathcal{A}_n$ and summing over
$n=1,\dots,N$ yields
\[
  Y_{N,\varphi}-\mathbb E_\xi[Y_{N,\varphi}\mid S]
  =\sum_{q=1}^k c_{\varphi,q}\,I_q(f_{q,S}).
\]
Since Lemma~\ref{lem:Yk-mean} shows that $\mathbb E_\xi[Y_{N,\varphi}\mid S]=\mathbb E[Y_{N,\varphi}]$, we have 
\begin{equation}\label{eq:Yk-chaos}
   Y_{N,\varphi}-\mathbb E[Y_{N,\varphi}]
  =\sum_{q=1}^k c_{\varphi,q}\,I_q(f_{q,S}),
  \qquad \mathbb P_S\text{-a.s.},
\end{equation}
as an identity in $L^2(\mathbb P_\xi)$.

\subsection{Variance growth: \texorpdfstring{$\sigma_{N,\varphi}^2\asymp N$}{variance of order N}}

To normalize the chaos expansion, we will show that $\sigma_{N,\varphi}^2$ grows linearly in $N$ in this subsection.

For $q\ge1$ and $t\ge1$, define the drift-averaged sequence $a_t^{(q)}=\mathbb E_S\big[Q_t(S_t)^q\big]$.

We call
\begin{equation}\label{eq:Poisson-Charlier-rank}
  r_\varphi
  =\min\big\{q\in\{1,\dots,k\}:c_{\varphi,q}\neq0\big\}
\end{equation}
the Poisson--Charlier rank of $\varphi$.
This is well defined because $c_{\varphi,k}=\beta_k\neq0$.
\begin{lemma}[Annealed variance]\label{lem:Yk-Var}
Under \eqref{eq:drift-assumption-Yk}, there exists $\mathfrak c_\varphi\in(0,\infty)$ such that
\[
  \sigma_{N,\varphi}^2=\operatorname{Var}(Y_{N,\varphi})=\mathfrak c_\varphi\,N+O(1),
  \qquad\text{and hence}\qquad \sigma_{N,\varphi}^2\asymp N.
\]
More precisely,
\begin{equation}\label{eq:ck-constant}
  \mathfrak c_\varphi=\sum_{q=1}^k c_{\varphi,q}^2\,q!\,\lambda^q\Big(1+2\sum_{t\ge1}a_t^{(q)}\Big).
\end{equation}
\end{lemma}

\begin{proof}
Since $\mathbb E_\xi[Y_{N,\varphi}\mid S]=N\,\mathbb E[\varphi(\mathrm{Poi}(\lambda))]$ is deterministic, the law of total
variance yields
\[
  \sigma_{N,\varphi}^2=\operatorname{Var}(Y_{N,\varphi})=\mathbb E_S\big[\operatorname{Var}_\xi(Y_{N,\varphi}\mid S)\big].
\]
By \eqref{eq:Yk-chaos} and orthogonality of Poisson chaoses,
\begin{equation}\label{eq:VarYk-condS}
  \operatorname{Var}_\xi(Y_{N,\varphi}\mid S)
  =\sum_{q=1}^k c_{\varphi,q}^2\,q!\,\|f_{q,S}\|_{L^2(\mu^{\otimes q})}^2.
\end{equation}
We compute $\|f_{q,S}\|_{L^2(\mu^{\otimes q})}^2$ explicitly. Expanding \eqref{eq:FqS-def} and using the product structure of
$\mu^{\otimes q}$, we obtain
\begin{align*}
  \|f_{q,S}\|_{L^2(\mu^{\otimes q})}^2
  &=\int_{\mathbb W^q}\Big(\sum_{n=1}^N \prod_{j=1}^q \mathbf 1_{\mathcal{A}_n}(w_j)\Big)^2\,\mu^{\otimes q}(\mathrm d\vec w) \\
  &=\sum_{n,m=1}^N \int_{\mathbb W^q}\prod_{j=1}^q \mathbf 1_{\mathcal{A}_n\cap \mathcal{A}_m}(w_j)\,\mu^{\otimes q}(\mathrm d\vec w)
   =\sum_{n,m=1}^N \mu(\mathcal{A}_n\cap \mathcal{A}_m)^q.
\end{align*}
Since $\mu(\mathcal{A}_n\cap \mathcal{A}_m)=\lambda Q_{|m-n|}(S_m-S_n)$, we obtain
\begin{equation}\label{F_q,s}
  \|f_{q,S}\|_{L^2(\mu^{\otimes q})}^2
  =\lambda^q\sum_{n,m=1}^N Q_{|m-n|}(S_m-S_n)^q
  =\lambda^q\Big(N+2\Sigma^{(q)}_S\Big),
\end{equation}
where
\[
  \Sigma^{(q)}_S:=\sum_{1\le n<m\le N}Q_{m-n}(S_m-S_n)^q.
\]
Taking $\mathbb E_S$ and using stationarity of increments ($S_{n+t}-S_n\stackrel{d}=S_t$), we obtain
\[
  \mathbb E_S[\Sigma^{(q)}_S]=\sum_{t=1}^{N-1}(N-t)\,a_t^{(q)}.
\]
Therefore,
\begin{equation}\label{eq:sigmak-expansion}
  \sigma_{N,\varphi}^2
  =\sum_{q=1}^k c_{\varphi,q}^2\,q!\,\lambda^q
    \Big(N+2\sum_{t=1}^{N-1}(N-t)a_t^{(q)}\Big).
\end{equation}

By Lemma~\ref{lem:moments-decay-Yk} with $m=q$, the sequence $a^{(q)}$ is summable and
$\sum_{t\ge1}t\,a_t^{(q)}<\infty$, and hence
\[
  \sum_{t=1}^{N-1}(N-t)a_t^{(q)}
  =N\sum_{t\ge1}a_t^{(q)}-\sum_{t\ge1}t\,a_t^{(q)}+o(1)
  =N\sum_{t\ge1}a_t^{(q)}+O(1).
\]
Plugging this into \eqref{eq:sigmak-expansion} yields $\sigma_{N,\varphi}^2=\mathfrak c_\varphi N+O(1)$ with $\mathfrak c_\varphi$ given by
\eqref{eq:ck-constant}. Since $a_t^{(q)}\ge0$ for all $q$ and $t$, every summand in
\eqref{eq:ck-constant} is nonnegative. Moreover,
$c_{\varphi,r_\varphi}\neq0$, and hence the $q=r_\varphi$ summand is
bounded below by
$c_{\varphi,r_\varphi}^2\,r_\varphi!\,\lambda^{r_\varphi}>0.$
Therefore $\mathfrak c_\varphi>0$.
\end{proof}

\subsection{Annealed Malliavin--Stein bound}

Our goal is to bound the Wasserstein distance between the law of $F_{N,\varphi}$ and the standard
Gaussian law under $\mathbb P$. Unlike the fixed-region functionals
$W_N(A)$ and $D_N(A)$, the variable $F_{N,\varphi}$ depends on two independent sources of
randomness: the Poisson field $\xi$ and the sampling walk $S$. Therefore, under the annealed law,
$F_{N,\varphi}$ is not a Poisson functional of $\eta$ alone, so the standard Poisson
Malliavin--Stein bound \eqref{SB2} cannot be applied directly.
The key observation is that, for a fixed realization of $S$, the expansion
\eqref{eq:Yk-chaos} is $\mathbb P_S$-a.s.\ a finite Poisson-chaos expansion in $L^2(\mathbb P_\xi)$. Hence the usual Poisson
Malliavin--Stein inequality applies conditionally on $S$. The only additional step is then to
average these conditional bounds over the sampling walk in order to recover an annealed estimate.
Throughout this subsection, the Malliavin operators act only in the Poisson direction, with $S$
regarded as fixed.

\begin{lemma}[Annealed distance bound via conditioning]\label{lem:annealed-by-conditional-Yk}
For any $N\ge1$,
\[
  d_\mathcal H\!\big(\mathcal L(F_{N,\varphi}),\mathcal L(Z)\big)
  \;\le\;
  \mathbb E_S\Big[d_\mathcal H\!\big(\mathcal L_\xi(F_{N,\varphi}\mid S),\mathcal L(Z)\big)\Big],
\]
where $Z\sim\mathcal N(0,1)$ is independent of $(\eta,S)$.
\end{lemma}

\begin{proof}
For any $h\in\mathcal H$, we have
\begin{align*}
  \Big|\mathbb E[h(F_{N,\varphi})]-\mathbb E[h(Z)]\Big|
  = & \Big|\mathbb E_S\big[\mathbb E_\xi[h(F_{N,\varphi})\mid S]-\mathbb E[h(Z)]\big]\Big|\\
  \le& \mathbb E_S\Big[\big|\mathbb E_\xi[h(F_{N,\varphi})\mid S]-\mathbb E[h(Z)]\big|\Big]\\
  \le& \mathbb E_S\Big[d_\mathcal H\!\big(\mathcal L_\xi(F_{N,\varphi}\mid S),\mathcal L(Z)\big)\Big].
\end{align*}
Taking the supremum over all $h\in\mathcal H$ yields the claim.
\end{proof}

For a fixed realization of $S$, the expansion
\eqref{eq:Yk-chaos} is $\mathbb P_S$-a.s.\ a finite Poisson-chaos expansion in
$L^2(\mathbb P_\xi)$. In view of
$\|f_{q,S}\|_{L^2(\mu^{\otimes q})}^2\le N^2\lambda^q$, it follows directly from the
definitions in Subsection~\ref{malliavincalculus} that, conditionally on $S$,
\[
  F_{N,\varphi}\in\operatorname{dom}(L)\subset\operatorname{dom}(D),
  \qquad
  L^{-1}F_{N,\varphi}\in\operatorname{dom}(D).
\]
Thus the Malliavin operators below are well defined pathwise, acting only
on the Poisson coordinate, with $S$ regarded as fixed. Applying
\eqref{De} and \eqref{L_inverse} pathwise, we obtain, for $\mu$-a.e.
$w\in\mathbb W$,
\begin{equation}\label{eq:DF-Yk}
  D_wF_{N,\varphi}
  =\frac{1}{\sigma_{N,\varphi}}\sum_{q=1}^k c_{\varphi,q}\,q\,I_{q-1}\big(f_{q,S}(w,\cdot)\big),
\end{equation}
and
\begin{equation}\label{eq:DLinv-Yk}
  -D_wL^{-1}F_{N,\varphi}
  =\frac{1}{\sigma_{N,\varphi}}\sum_{q=1}^k c_{\varphi,q}\,I_{q-1}\big(f_{q,S}(w,\cdot)\big).
\end{equation}
Applying the conditional Poisson Malliavin--Stein bound and then averaging over $S$ via
Lemma~\ref{lem:annealed-by-conditional-Yk} yields
\begin{equation}\label{eq:MS-Yk}
  d_W(F_{N,\varphi},Z)
  \le
  \sqrt{\frac2\pi}\,
  \mathbb E\Big[\Big|1-\langle DF_{N,\varphi},-DL^{-1}F_{N,\varphi}\rangle_{L^2(\mu)}\Big|\Big]
  +
  \mathbb E\int_{\mathbb W} |D_wF_{N,\varphi}|^2\,|D_wL^{-1}F_{N,\varphi}|\,\mu(\mathrm dw).
\end{equation}

Set
\begin{equation}\label{eq:Theta_{N,k}}
  \Theta_{N,\varphi}
  =\langle DF_{N,\varphi},-DL^{-1}F_{N,\varphi}\rangle_{L^2(\mu)},
\end{equation}
and define
\[
 \mathrm T_1^{(\varphi)}=\mathbb E\big[|1-\Theta_{N,\varphi}|\big],
  \qquad
  \mathrm T_2^{(\varphi)}=\mathbb E\int_{\mathbb W}|D_wF_{N,\varphi}|^2\,|D_wL^{-1}F_{N,\varphi}|\,\mu(\mathrm dw).
\]
Then \eqref{eq:MS-Yk} becomes
\begin{equation}\label{eq:MS-case1}
  d_W(F_{N,\varphi},Z)
  \le
  \sqrt{\frac2\pi}\,\mathrm T_1^{(\varphi)}+\mathrm T_2^{(\varphi)}.
\end{equation}
In the remainder of this subsection, we first consider the basic observable $ \varphi(x)=x$,
which already captures the main route of the proof in a particularly transparent form. In this case
only the first Poisson chaos appears, whereas the general polynomial case follows the same strategy
with a more involved finite-chaos expansion.
%==============================================================================
\subsubsection{A warm-up: the identity observable \texorpdfstring{$\varphi(x)=x$}{phi(x) = x}}\label{subsubsec:Yk-case1}
%==============================================================================

In the linear case $\varphi(x)=x$, only the first Poisson chaos appears. More precisely,
$c_{x,1}=1$ and
\[
  Y_{N,x}-\mathbb E[Y_{N,x}]=I_1(f_{1,S}),
  \qquad
  f_{1,S}(w)=\sum_{n=1}^N \mathbf 1_{\mathcal A_n}(w).
\]
Moreover, by Lemma~\ref{lem:Yk-Var},
\begin{equation}\label{eq:sigmaN1-asymp}
  \sigma_{N,x}^2=\mathfrak c_x N+O(1),\qquad \sigma_{N,x}^2\asymp N.
\end{equation}
Hence
\begin{equation}\label{eq:DF-case1}
  F_{N,x}=\frac{I_1(f_{1,S})}{\sigma_{N,x}},
  \qquad
  D_wF_{N,x}=\frac{f_{1,S}(w)}{\sigma_{N,x}},
  \qquad
  -D_wL^{-1}F_{N,x}=D_wF_{N,x}.
\end{equation}

From \eqref{eq:Theta_{N,k}}, we have
\[
  \Theta_{N,x}=\langle DF_{N,x},-DL^{-1}F_{N,x}\rangle_{L^2(\mu)}
          =\int_{\mathbb W}|D_wF_{N,x}|^2\,\mu(\mathrm dw).
\]
By \eqref{F_q,s}, this yields
\begin{equation}\label{eq:ThetaS-case1}
  \Theta_{N,x}
  =\frac{1}{\sigma_{N,x}^2}\sum_{n,m=1}^N\mu(\mathcal{A}_n\cap \mathcal{A}_m)
  =\frac{1}{\sigma_{N,x}^2}\Big(\lambda N+2\lambda\Sigma^{(1)}_S\Big).
\end{equation}
Note that $\mathbb E[\Theta_{N,x}]=1$ since $\mathbb E[\lambda N+2\lambda\Sigma^{(1)}_S]=\sigma_{N,x}^2$ by \eqref{eq:sigmak-expansion}.

To control $\mathrm T_1^{(x)}$, we need a variance bound for the correlation sum $\Sigma^{(1)}_S$.
The following estimate will also be used later in the general polynomial case.

\begin{lemma}\label{lem:VarSigma-casek}
Fix $q\in\{1,\dots,k\}$. Under \eqref{eq:drift-assumption-Yk}, there exists $C:=C(\mathfrak p,q)\in (0,\infty)$ such that for all $N\ge1$,
\[
  \operatorname{Var}_S(\Sigma^{(q)}_S)\le C\,N,
\] where $\Sigma^{(q)}_S=\sum_{1\le n<m\le N}Q_{m-n}(S_m-S_n)^q$.
\end{lemma}

\begin{proof}
Set $X^{(q)}_{n,t}=Q_t(S_{n+t}-S_n)^q$ and $U_t=\sum_{n=1}^{N-t}X^{(q)}_{n,t}$, so that
$\Sigma^{(q)}_S=\sum_{t=1}^{N-1}U_t$.
Let $\omega=(\Delta_i)_{i\ge1}$ be the increment sequence. Let $I(n,t):=\{n+1,\dots,n+t\}$.
Then $X^{(q)}_{n,t}$ is measurable with respect to
$\sigma(\Delta_i:i\in I(n,t))$.
Hence, if $I(n,t)\cap I(m,s)=\emptyset$, then
$X^{(q)}_{n,t}$ and $X^{(q)}_{m,s}$ are independent, and $\operatorname{Cov}(X^{(q)}_{n,t},X^{(q)}_{m,s})=0$.

Fix $t,s\ge1$. If $I(n,t)\cap I(m,s)\neq\emptyset$, then necessarily
$m\in\{n-s+1,\dots,n+t-1\}$, so for each $n$ there are at most $t+s$ indices $m$ producing an
overlap. Consequently,
\begin{equation}\label{eq:overlap-count-Yk}
  \#\{(n,m): 1\le n\le N-t,\ 1\le m\le N-s,\ I(n,t)\cap I(m,s)\neq\emptyset\}
  \le N(t+s).
\end{equation}

Let $b_t^{(q)}=\mathbb E_S[X^{(q)}_{1,t}{}^2]=\mathbb E_S[Q_t(S_t)^{2q}]$. By the Cauchy--Schwarz inequality,
\[
  |\operatorname{Cov}(X^{(q)}_{n,t},X^{(q)}_{m,s})|
  \le \sqrt{\mathbb E[X^{(q)}_{n,t}{}^2]\mathbb E[X^{(q)}_{m,s}{}^2]}
  =\sqrt{b_t^{(q)}b_s^{(q)}}.
\]
Therefore, using \eqref{eq:overlap-count-Yk}, we obtain
\[
  |\operatorname{Cov}(U_t,U_s)|
  \le \sum_{\substack{1\le n\le N-t\\1\le m\le N-s\\ I(n,t)\cap I(m,s)\neq\emptyset}}
       |\operatorname{Cov}(X^{(q)}_{n,t},X^{(q)}_{m,s})|
  \le N(t+s)\sqrt{b_t^{(q)}b_s^{(q)}}.
\]
Finally,
\[
  \operatorname{Var}_S(\Sigma^{(q)}_S)
  =\operatorname{Var}_S\Big(\sum_{t=1}^{N-1}U_t\Big)
  \le \sum_{t,s=1}^{N-1}|\operatorname{Cov}(U_t,U_s)|
  \le N\sum_{t,s\ge1}(t+s)\sqrt{b_t^{(q)}b_s^{(q)}}.
\]
By Lemma~\ref{lem:moments-decay-Yk} with $m=2q$, we have $b_t^{(q)}\le C_1 t^{-q}\mathrm{e}^{-c_1t}$, and hence
$\sqrt{b_t^{(q)}}\le C_1 t^{-q/2}\mathrm{e}^{-c_1t/2}$. Consequently, both $\sum_{t\ge 1}\sqrt{b_t^{(q)}}$ and
$\sum_{t\ge 1} t\sqrt{b_t^{(q)}}$ converge, which yields $\operatorname{Var}_S(\Sigma^{(q)}_S)\le C N$.
\end{proof}

With Lemma~\ref{lem:VarSigma-casek} in hand, we return to $\Theta_{N,x}$. Using \eqref{eq:ThetaS-case1} and the Cauchy--Schwarz inequality, we obtain
\begin{equation}\label{eq:term1-case1}
  \mathbb E\big[|1-\Theta_{N,x}|\big]
  \le \sqrt{\operatorname{Var}(\Theta_{N,x})}
  =\frac{2\lambda}{\sigma_{N,x}^2}\sqrt{\operatorname{Var}_S(\Sigma^{(1)}_S)}
  \le \frac{C}{\sigma_{N,x}^2}\sqrt{N}
  \le C\,N^{-1/2},
\end{equation}
where the last step uses $\sigma_{N,x}^2\asymp N$ from \eqref{eq:sigmaN1-asymp}.

We now turn to $\mathrm T_2^{(x)}$. By \eqref{eq:DF-case1} and Tonelli's theorem,
\begin{align}
  \int_{\mathbb W}\mathbb E\big[|D_wF_{N,x}|^3\big]\,\mu(\mathrm dw)
  &=\frac{1}{\sigma_{N,x}^3}\int_{\mathbb W}\mathbb E\Big[\Big(\sum_{n=1}^N\mathbf 1_{\mathcal{A}_n}(w)\Big)^3\Big]\mu(\mathrm dw)\notag\\
  &=\frac{1}{\sigma_{N,x}^3}\sum_{i,j,\ell=1}^N \mathbb E\,\mu(\mathcal{A}_i\cap \mathcal{A}_j\cap \mathcal{A}_\ell). \label{eq:term2-case1-start}
\end{align}
We estimate the sum of triple intersections by splitting according to the coincidence pattern of $(i,j,\ell)$.

   {\rm(i)} \textbf{All equal.} This contributes $\frac{N\lambda}{\sigma_{N,x}^3}$ since $\mu(\mathcal{A}_n)=\lambda$.

{\rm(ii)} \textbf{Exactly two equal.} Absorbing the permutation factor into $C$, we bound this contribution by
\[
  \frac{C}{\sigma_{N,x}^3}\sum_{1\le n<m\le N}\mathbb E\,\mu(\mathcal{A}_n\cap \mathcal{A}_m)
  =\frac{C\lambda}{\sigma_{N,x}^3}\mathbb E\Big[\sum_{1\le n<m\le N}Q_{m-n}(S_m-S_n)\Big]
  =\frac{C\lambda}{\sigma_{N,x}^3}\mathbb E[\Sigma^{(1)}_S].
\]
This is $O(N/\sigma_{N,x}^3)$, since $\mathbb E[\Sigma^{(1)}_S]=\sum_{t=1}^{N-1}(N-t)a_t^{(1)}=O(N)$ by the summability of $a_t^{(1)}$.

{\rm(iii)} \textbf{All distinct indices.} For $n<m<\ell$,
\[
  \mu(\mathcal{A}_n\cap \mathcal{A}_m\cap \mathcal{A}_\ell)
  =\lambda\,Q_{m-n}(S_m-S_n)\,Q_{\ell-m}(S_\ell-S_m).
\]
By independence of disjoint increment blocks of $S$, it follows that
\[
  \mathbb E\big[Q_{m-n}(S_m-S_n)\,Q_{\ell-m}(S_\ell-S_m)\big]
  =a_{m-n}^{(1)}\,a_{\ell-m}^{(1)}.
\]
Therefore,
\[
  \sum_{1\le n<m<\ell\le N}\mathbb E\,\mu(\mathcal{A}_n\cap \mathcal{A}_m\cap \mathcal{A}_\ell)
  =\lambda\sum_{t_1,t_2\ge1}(N-t_1-t_2)_+\,a_{t_1}^{(1)}a_{t_2}^{(1)}
  \le \lambda N\Big(\sum_{t\ge1}a_t^{(1)}\Big)^2
  =O(N),
\]
where we used Lemma~\ref{lem:conv-count-Yk} with $M=2$.
Substituting (i)--(iii) into \eqref{eq:term2-case1-start} and noting $\sigma_{N,x}^2\asymp N$, we have
\begin{equation}\label{eq:term2-case1}
  \int_{\mathbb W}\mathbb E\big[|D_wF_{N,x}|^3\big]\,\mu(\mathrm dw)
  \le C\,\frac{N}{\sigma_{N,x}^3}
  \le C\,N^{-1/2}.
\end{equation}

Finally, substituting \eqref{eq:term1-case1} and \eqref{eq:term2-case1}
into \eqref{eq:MS-case1}, we obtain
\[
  d_W(F_{N,x},Z)\le C\,N^{-1/2}.
\]
This proves an annealed CLT with Wasserstein rate $N^{-1/2}$ for $Y_{N,x}$ under drifted sampling.

%==============================================================================
%==============================================================================
\subsubsection{General polynomial observables}\label{subsubsec:Yk-case2}
The proof for a general polynomial observable follows the same overall scheme, but the finite chaos
expansion now involves several orders simultaneously, which leads to additional contraction terms.

We estimate $\mathrm T_1^{(\varphi)}$ and $\mathrm T_2^{(\varphi)}$ separately. The first term is controlled via a
chaos decomposition of $\Theta_{N,\varphi}$ and a total-variance argument, whereas the second is reduced to
third moments of lower-order multiple integrals and then bounded through contraction estimates.

\noindent\textbf{Estimate of $\mathrm T_1^{(\varphi)}$}

Conditionally on \(S\), the random variable \(F_{N,\varphi}\) is centered and has a finite
Poisson-chaos expansion. Hence, by the standard Poisson covariance identity, which follows
from Lemma~2.1(i)--(ii) in \cite{eichelsbacher2014new},
\[
  \mathbb E_\xi[\Theta_{N,\varphi}\mid S]
  =\mathbb E_\xi[F_{N,\varphi}^2\mid S]
  =\operatorname{Var}_\xi(F_{N,\varphi}\mid S)
  =\frac{V_S}{\sigma_{N,\varphi}^2},
\]
where $V_S:=\operatorname{Var}_\xi(Y_{N,\varphi}\mid S)$.
Taking expectation with respect to $S$ and using
$\mathbb E_S[V_S]=\operatorname{Var}(Y_{N,\varphi})=\sigma_{N,\varphi}^2$, we obtain $\mathbb E[\Theta_{N,\varphi}]=1$.
Hence by the Cauchy--Schwarz inequality
\begin{equation}\label{eq:term1-cs-Yk}
  \mathbb E\big[|1-\Theta_{N,\varphi}|\big]\le \sqrt{\operatorname{Var}(\Theta_{N,\varphi})}.
\end{equation}
Thus it suffices to show $\operatorname{Var}(\Theta_{N,\varphi})=O(N^{-1})$.

To bound $\operatorname{Var}(\Theta_{N,\varphi})$, we first decompose $\Theta_{N,\varphi}$ into Poisson chaos conditionally on $S$.
Substituting \eqref{eq:DF-Yk}--\eqref{eq:DLinv-Yk} into \eqref{eq:Theta_{N,k}} gives
\[
  \Theta_{N,\varphi}
  =\frac{1}{\sigma_{N,\varphi}^2}\sum_{q_1,q_2=1}^k c_{\varphi,q_1}c_{\varphi,q_2}\,q_1
     \int_{\mathbb W} I_{q_1-1}\!\big(f_{q_1,S}(w,\cdot)\big)\,
               I_{q_2-1}\!\big(f_{q_2,S}(w,\cdot)\big)\,\mu(\mathrm dw).
\]
For each fixed $w$, we apply the product formula \eqref{mutiIn} to the product of the two multiple
integrals of orders $q_1-1$ and $q_2-1$. This produces a finite sum of multiple integrals whose
kernels are contractions of $f_{q_1,S}(w,\cdot)$ and $f_{q_2,S}(w,\cdot)$. Accordingly, for
$r\in\{0,\dots,\min(q_1-1,q_2-1)\}$ and $\ell\in\{0,\dots,r\}$ we define
\begin{equation}\label{eq:Hrsl-def-Yk}
  H_{q_1,q_2,r,\ell,S}
  =
  \int_{\mathbb W}\Big(
    f_{q_1,S}(w,\cdot)\,\widetilde{\star}_r^\ell\,f_{q_2,S}(w,\cdot)
  \Big)\,\mu(\mathrm dw).
\end{equation}
Let
\[
  \alpha_{q_1,q_2,r,\ell}
  =
  c_{\varphi,q_1}\,c_{\varphi,q_2}\,q_1\;
  r!\binom{q_1-1}{r}\binom{q_2-1}{r}\binom{r}{\ell},
\]
and for each $ \mathfrak{m}\in\{0,1,\dots,2k-2\}$ introduce the index set
\[
  \mathcal I_\mathfrak{m}
  =
  \Big\{(q_1,q_2,r,\ell):\ 
  1\le q_1,q_2\le k,\ 
  0\le r\le \min(q_1-1,q_2-1),\
  0\le \ell\le r,\
  q_1+q_2-2-r-\ell=\mathfrak{m}
  \Big\}.
\]
Because each $f_{q,S}(w,\cdot)$ is a finite sum of indicator tensors of sets
with finite $\mu$-measures, the map
\[
  w\longmapsto
  f_{q_1,S}(w,\cdot)\,
  \widetilde\star_r^\ell\,
  f_{q_2,S}(w,\cdot)
\]
is Bochner integrable in the relevant $L^2$ space. Since
\[
  I_{\mathfrak m}:
  L_s^2(\mu^{\otimes\mathfrak m})
  \longrightarrow
  L^2(\mathbb P_\xi)
\]
is continuous and linear, it commutes with this Bochner integral.
Together with the finite product expansion, this yields the finite chaos
decomposition
\begin{equation}\label{eq:Theta-chaos-Yk}
  \Theta_{N,\varphi}
  =
  \frac{V_S}{\sigma_{N,\varphi}^2}
  +\frac{1}{\sigma_{N,\varphi}^2}\sum_{\mathfrak{m}=1}^{2k-2} I_\mathfrak{m}(h_{\mathfrak{m},S}),
\end{equation}
where the zeroth chaos term corresponds precisely to the tuples
$(q_1,q_2,r,\ell)=(q,q,q-1,q-1)$, for which
\[
H_{q,q,q-1,q-1,S}=\|f_{q,S}\|_{L^2(\mu^{\otimes q})}^2,
\]
and hence the zeroth chaos contribution equals
\[
\sum_{q=1}^k c_{\varphi,q}^2 q!\,
\|f_{q,S}\|_{L^2(\mu^{\otimes q})}^2
=V_S,
\] 
and for each $\mathfrak{m}=1,\dots,2k-2$, the kernel
$h_{\mathfrak{m},S}\in L^2(\mu^{\otimes \mathfrak{m}})$ is given explicitly by
\[
  h_{\mathfrak{m},S}
  =
  \sum_{(q_1,q_2,r,\ell)\in\mathcal I_\mathfrak{m}}
  \alpha_{q_1,q_2,r,\ell}\,
  H_{q_1,q_2,r,\ell,S}.
\]
In particular, since the number of terms depends only on $\varphi$, there exists $C:=C(\varphi,\lambda)\in (0,\infty)$ such that for
every realization of $S$,
\begin{equation}\label{eq:VarTheta-cond-bound-Yk}
  \operatorname{Var}_\xi(\Theta_{N,\varphi}\mid S)
\le
\frac{C}{\sigma_{N,\varphi}^4}
\sum_{\mathfrak{m}=1}^{2k-2}\;
\sum_{(q_1,q_2,r,\ell)\in\mathcal I_\mathfrak{m}}
\|H_{q_1,q_2,r,\ell,S}\|_{L^2(\mu^{\otimes \mathfrak{m}})}^2.
\end{equation}
Taking $\mathbb E_\xi[\cdot\mid S]$ in \eqref{eq:Theta-chaos-Yk} gives the conditional mean identity
\begin{equation}\label{eq:ETheta-cond-Yk}
  \mathbb E_\xi[\Theta_{N,\varphi}\mid S]=\frac{V_S}{\sigma_{N,\varphi}^2}.
\end{equation}
Furthermore, by the law of total variance, we have
\begin{equation}\label{eq:VarTheta-decomp-Y2-in-Yk}
  \operatorname{Var}(\Theta_{N,\varphi})
  =\mathbb E\big[\operatorname{Var}_\xi(\Theta_{N,\varphi}\mid S)\big] + \operatorname{Var}\!\Big(\frac{V_S}{\sigma_{N,\varphi}^2}\Big).
\end{equation}

To control $\operatorname{Var}(\Theta_{N,\varphi})$, we need the following bound on the fluctuations of
$V_S$ under $\mathbb P_S$. 

\begin{lemma}\label{lem:VarVS-Yk}
Under \eqref{eq:drift-assumption-Yk}, there exists $C:=C(\varphi,\lambda,\mathfrak p)\in (0,\infty)$ such that for all $N\ge1$,
\[
  \operatorname{Var}_S(V_S)\le C\,N.
\]
\end{lemma}

\begin{proof}
Recall that
\[
  V_S
  =\sum_{q=1}^k c_{\varphi,q}^2\,q!\,\lambda^q\big(N+2\Sigma^{(q)}_S\big),
  \qquad
  \Sigma^{(q)}_S=\sum_{1\le n<m\le N}Q_{m-n}(S_m-S_n)^q.
\]
Therefore,
\[
  \operatorname{Var}_S(V_S)\le C\sum_{q=1}^k \operatorname{Var}_S(\Sigma^{(q)}_S)\le C N
\]
by Lemma~\ref{lem:VarSigma-casek}.
\end{proof}

Before bounding $\operatorname{Var}_\xi(\Theta_{N,\varphi}\mid S)$,
we first introduce a technical lemma that will be used in the proofs of
Lemmas~\ref{lem:Hrs-L2-Yk} and~\ref{lem:TqN-ON-fixed} below.

\begin{lemma}[Four-time connected-intersection sum estimate]
\label{lem:four-time-connected-Yk}
Fix $K\in\mathbb N$ and let
\[
  \boldsymbol\gamma
  =(\gamma_{12},\gamma_{34},\gamma_{1234},\gamma_{13},\gamma_{24})
  \in\{0,\ldots,K\}^5.
\]
Suppose that either
\[
  \gamma_{1234}\ge1,
\]
or
\[
  \gamma_{12},\gamma_{34}\ge1
  \qquad\text{and}\qquad
  \gamma_{13}+\gamma_{24}\ge1.
\]
For $\boldsymbol n=(n_1,n_2,n_3,n_4)\in\{1,\ldots,N\}^4$, write
\[
  M_{ij}=\mu(\mathcal A_{n_i}\cap\mathcal A_{n_j}),
  \qquad
  M_{1234}=
  \mu\Big(\bigcap_{i=1}^4\mathcal A_{n_i}\Big),
\]
and set
\[
  \mathcal G_{\boldsymbol\gamma}(\boldsymbol n)
  =
  M_{12}^{\gamma_{12}}
  M_{34}^{\gamma_{34}}
  M_{1234}^{\gamma_{1234}}
  M_{13}^{\gamma_{13}}
  M_{24}^{\gamma_{24}},
\]
where factors with exponent zero are interpreted as $1$. Under \eqref{eq:drift-assumption-Yk}, there exists $C:=C(K,\lambda,\mathfrak p)\in(0,\infty)$ such that, for every $N\ge1$,
\begin{equation}\label{eq:four-time-sum-Yk}
  \sum_{n_1,n_2,n_3,n_4=1}^N
  \mathbb E_S\big[
    \mathcal G_{\boldsymbol\gamma}(\boldsymbol n)
  \big]
  \le CN.
\end{equation}
\end{lemma}

\begin{proof}
Let
\[
  \Gamma=
  \gamma_{12}+\gamma_{34}+\gamma_{1234}
  +\gamma_{13}+\gamma_{24}.
\]
Since every intersection appearing above is contained in one of the
sets $\mathcal A_{n_i}$ and $\mu(\mathcal A_{n_i})=\lambda$, we have
\[
  0\le \frac{M_{ij}}{\lambda}\le1,
  \qquad
  0\le \frac{M_{1234}}{\lambda}\le1.
\]

Let
\[
  t_1<\cdots<t_d,\qquad 1\le d\le4,
\]
be the distinct values among $n_1,n_2,n_3,n_4$. We first show that
there exist constants $C,c>0$, depending only on
$K,\lambda,\mathfrak p$, such that
\[
  \mathbb E_S\big[
    \mathcal G_{\boldsymbol\gamma}(\boldsymbol n)
  \big]
  \le
  C \mathrm{e}^{-c(t_d-t_1)}.
\]
When $d=1$, this is immediate since
$\mathcal G_{\boldsymbol\gamma}(\boldsymbol n)=\lambda^\Gamma$.

Suppose now that $d\ge2$. Choose $j\in\{1,\ldots,d-1\}$ such that
\[
  g:=t_{j+1}-t_j
  =\max_{1\le h\le d-1}(t_{h+1}-t_h),
\]
and partition the four labels according to this gap:
\[
  L:=\{i:n_i\le t_j\},
  \qquad
  R:=\{i:n_i\ge t_{j+1}\}.
\]
Both $L$ and $R$ are nonempty.

We claim that some intersection factor with a positive exponent contains
a label from $L$ and a label from $R$. If $\gamma_{1234}\ge1$, this is
immediate from the fourfold intersection factor. Otherwise,
$\gamma_{12},\gamma_{34}\ge1$ and
$\gamma_{13}+\gamma_{24}\ge1$. Hence the positive pair factors contain
the edges
\[
  \{1,2\},\qquad \{3,4\},
\]
together with at least one of $\{1,3\}$ and $\{2,4\}$. These edges form
a connected graph on $\{1,2,3,4\}$, so at least one of them must cross
the nontrivial partition $L\cup R$.

Choose such a positive-exponent factor and labels
$i\in L$, $h\in R$ contained in it. Put $u=n_h-n_i.$
Then
\[
  u\ge g
  \ge \frac{t_d-t_1}{d-1}
  \ge \frac{t_d-t_1}{3}.
\]
Let $M_*$ denote the corresponding intersection mass. Since the
intersection defining $M_*$ is contained in
$\mathcal A_{n_i}\cap\mathcal A_{n_h}$, and since all normalized
intersection factors are bounded by $1$, we obtain
\[
  \mathcal G_{\boldsymbol\gamma}(\boldsymbol n)
  \le
  \lambda^\Gamma\frac{M_*}{\lambda}
  \le
  \lambda^\Gamma
  \frac{\mu(\mathcal A_{n_i}\cap\mathcal A_{n_h})}{\lambda}.
\]
Note that
$\mu(\mathcal A_{n_i}\cap\mathcal A_{n_h})=\lambda Q_u(S_{n_h}-S_{n_i}).$
Therefore, by stationary increments of $S$ and
Lemma~\ref{lem:moments-decay-Yk} with $m=1$,
\[
  \mathbb E_S\big[
    \mathcal G_{\boldsymbol\gamma}(\boldsymbol n)
  \big]
  \le
  C\,\mathbb E_S[Q_u(S_u)]
  \le
  C \mathrm{e}^{-c_0u}
  \le
  C \mathrm{e}^{-(c_0/3)(t_d-t_1)}.
\]
Thus the claimed exponential bound in the total time span holds
uniformly over all admissible exponent vectors.

It remains to sum over the time indices. For each
$d\in\{2,3,4\}$ and each fixed ordered set
$t_1<\cdots<t_d$, there are at most $d^4$ assignments of the four
labels $n_1,\ldots,n_4$ to these times. Writing
\[
  a_h=t_{h+1}-t_h\ge1,
  \qquad h=1,\ldots,d-1,
\]
we therefore obtain
\begin{align*}
 \sum_{\substack{1\le n_1,n_2,n_3,n_4\le N\\
                  |\{n_1,n_2,n_3,n_4\}|=d}}
  \mathbb E_S\big[
    \mathcal G_{\boldsymbol\gamma}(\boldsymbol n)
  \big]
 &\le
 C\sum_{a_1,\ldots,a_{d-1}\ge1}
   (N-a_1-\cdots-a_{d-1})_+
   \mathrm{e}^{-c(a_1+\cdots+a_{d-1})}
\\
 &\le
 CN\left(\sum_{a\ge1}\mathrm{e}^{-ca}\right)^{d-1}
 \le CN.
\end{align*}
When $d=1$, all four indices coincide and there are only $N$
configurations, each contributing at most a constant depending on
$K$ and $\lambda$. Hence this contribution is also bounded by $CN$.
Summing over $d=1,2,3,4$ proves
\eqref{eq:four-time-sum-Yk}.
\end{proof}

We now provide the $L^2$-estimate for the contraction kernels introduced in \eqref{eq:Hrsl-def-Yk},
which is the key analytic input for bounding $\operatorname{Var}_\xi(\Theta_{N,\varphi}\mid S)$.
\begin{lemma}\label{lem:Hrs-L2-Yk}
Fix $\mathfrak m\ge1$ and $(q_1,q_2,r,\ell)\in\mathcal I_\mathfrak m$. Under
\eqref{eq:drift-assumption-Yk}, there exists
$C:=C(k,\lambda,\mathfrak p)\in(0,\infty)$ such that for all $N\ge1$,
\[
  \mathbb E_S\Big[
    \big\|H_{q_1,q_2,r,\ell,S}\big\|_{L^2(\mu^{\otimes\mathfrak m})}^2
  \Big]
  \le C N.
\]
\end{lemma}

\begin{proof}
Recall that, for fixed $w\in\mathbb W$,
\begin{equation}\label{eq:FqS-w-explicit-rewrite}
  f_{q,S}(w;u_1,\dots,u_{q-1})
  =\sum_{n=1}^N \mathbf 1_{\mathcal A_n}(w)
    \prod_{j=1}^{q-1}\mathbf 1_{\mathcal A_n}(u_j),
\end{equation}
which is symmetric in $(u_1,\dots,u_{q-1})$. We use the convention
$\mathbf 1_A^{\otimes0}=1$, and those factors with exponent zero are omitted.

For $n,m\in\{1,\dots,N\}$, a direct inspection of the definition of $\star_r^\ell$ gives
\[
  \mathbf 1_{\mathcal A_n}^{\otimes(q_1-1)}
  \star_r^\ell
  \mathbf 1_{\mathcal A_m}^{\otimes(q_2-1)}
  =
  \mu(\mathcal A_n\cap\mathcal A_m)^\ell
  \mathbf 1_{\mathcal A_n\cap\mathcal A_m}^{\otimes(r-\ell)}
  \otimes \mathbf 1_{\mathcal A_n}^{\otimes(q_1-1-r)}
  \otimes \mathbf 1_{\mathcal A_m}^{\otimes(q_2-1-r)}.
\]
Inserting \eqref{eq:FqS-w-explicit-rewrite} into the definition of
$H_{q_1,q_2,r,\ell,S}$, using bilinearity of $\star_r^\ell$, and integrating with respect to $\mu(\mathrm dw)$,
we obtain the unsymmetrized kernel
\[
  \bar H_{q_1,q_2,r,\ell,S}
  :=\int_{\mathbb W}
  \big(f_{q_1,S}(w,\cdot)\star_r^\ell f_{q_2,S}(w,\cdot)\big)\,\mu(\mathrm dw)
\]
in the form
\begin{equation}\label{eq:Hbar-explicit-rewrite}
  \bar H_{q_1,q_2,r,\ell,S}
  =\sum_{n,m=1}^N \mu(\mathcal A_n\cap\mathcal A_m)^{\ell+1}
  \mathbf 1_{\mathcal A_n\cap\mathcal A_m}^{\otimes(r-\ell)}
  \otimes \mathbf 1_{\mathcal A_n}^{\otimes(q_1-1-r)}
  \otimes \mathbf 1_{\mathcal A_m}^{\otimes(q_2-1-r)}.
\end{equation}
By linearity of canonical symmetrization, $H_{q_1,q_2,r,\ell,S}$ is the symmetrization of
$\bar H_{q_1,q_2,r,\ell,S}$. Since symmetrization is an $L^2$-contraction, it is enough to prove
\[
  \mathbb E_S\big[
    \|\bar H_{q_1,q_2,r,\ell,S}\|_{L^2(\mu^{\otimes\mathfrak m})}^2
  \big]
  \le C N.
\]

Expanding the square of \eqref{eq:Hbar-explicit-rewrite} and using the product structure of
$\mu^{\otimes\mathfrak m}$ gives
\begin{equation}\label{eq:Tnm-def-rewrite}
  \|\bar H_{q_1,q_2,r,\ell,S}\|_{L^2(\mu^{\otimes\mathfrak m})}^2
  =\sum_{n,m,n',m'=1}^N \mathcal T(n,m,n',m'),
\end{equation}
where
\begin{align*}
  \mathcal T(n,m,n',m')
  :=&\;
  \mu(\mathcal A_n\cap\mathcal A_m)^{\ell+1}
  \mu(\mathcal A_{n'}\cap\mathcal A_{m'})^{\ell+1}\\
  &\times
  \mu(\mathcal A_n\cap\mathcal A_m\cap
       \mathcal A_{n'}\cap\mathcal A_{m'})^{r-\ell}
  \mu(\mathcal A_n\cap\mathcal A_{n'})^{q_1-1-r}
  \mu(\mathcal A_m\cap\mathcal A_{m'})^{q_2-1-r}.
\end{align*}

Apply Lemma~\ref{lem:four-time-connected-Yk} with $K=k$,
$\boldsymbol n=(n,m,n',m')$, and
\[
  (\gamma_{12},\gamma_{34},\gamma_{1234},\gamma_{13},\gamma_{24})
  =(\ell+1,\ell+1,r-\ell,q_1-1-r,q_2-1-r).
\]
All five exponents lie in $\{0,\ldots,k\}$, and
$\gamma_{12},\gamma_{34}\ge1$. Moreover,
\[
  \gamma_{1234}+\gamma_{13}+\gamma_{24}
  =(r-\ell)+(q_1-1-r)+(q_2-1-r)
  =\mathfrak m\ge1.
\]
Hence either $\gamma_{1234}\ge1$, or
$\gamma_{13}+\gamma_{24}\ge1$. Lemma~\ref{lem:four-time-connected-Yk} therefore gives
\[
  \sum_{n,m,n',m'=1}^N\mathbb E_S[\mathcal T(n,m,n',m')]
  \le CN,
\]
including all coincident-time configurations. Together with
\eqref{eq:Tnm-def-rewrite} and the contraction property of symmetrization,
this yields
\[
  \mathbb E_S\big[
    \|H_{q_1,q_2,r,\ell,S}\|_{L^2(\mu^{\otimes\mathfrak m})}^2
  \big]
  \le \mathbb E_S\big[
    \|\bar H_{q_1,q_2,r,\ell,S}\|_{L^2(\mu^{\otimes\mathfrak m})}^2
  \big]\le CN.
\]
\end{proof}
We can now close the estimate for $\mathrm T_1^{(\varphi)}$.
Using \eqref{eq:VarTheta-cond-bound-Yk} and Lemma~\ref{lem:Hrs-L2-Yk}, we obtain
\[
  \mathbb E\big[\operatorname{Var}_\xi(\Theta_{N,\varphi}\mid S)\big]\le \frac{C}{\sigma_{N,\varphi}^4}\,N.
\]
By \eqref{eq:ETheta-cond-Yk} and Lemma~\ref{lem:VarVS-Yk},
\[
  \operatorname{Var}_S\big(\mathbb E_\xi[\Theta_{N,\varphi}\mid S]\big)
  =\operatorname{Var}_S\Big(\frac{V_S}{\sigma_{N,\varphi}^2}\Big)
  =\frac{1}{\sigma_{N,\varphi}^4}\operatorname{Var}_S(V_S)
  \le \frac{C}{\sigma_{N,\varphi}^4}\,N.
\]
Therefore, by \eqref{eq:VarTheta-decomp-Y2-in-Yk}
 and Lemma~\ref{lem:Yk-Var}, we have
\begin{equation}\label{eq:term1-var}
  \operatorname{Var}(\Theta_{N,\varphi})\le \frac{C}{\sigma_{N,\varphi}^4}\,N \le C/N.
\end{equation} 
Substituting \eqref{eq:term1-var} into \eqref{eq:term1-cs-Yk} gives
\begin{equation}\label{eq:term1-final-Yk}
  \mathrm T_1^{(\varphi)}\le C\,N^{-1/2}.
\end{equation}

\noindent\textbf{Estimate of $\mathrm T_2^{(\varphi)}$}

We now bound $\mathrm T_2^{(\varphi)}$ in \eqref{eq:MS-Yk}. By the elementary inequality
$|\sum_{j=1}^m x_j|^3\le m^2\sum_{j=1}^m |x_j|^3$ together with the finiteness of the sums in
\eqref{eq:DF-Yk}--\eqref{eq:DLinv-Yk}, there exists $C:=C(\varphi,\lambda)\in (0,\infty)$ such that for all $w$,
\begin{align*}
  |D_wF_{N,\varphi}|^2\,|D_wL^{-1}F_{N,\varphi}|
  &\le \frac{C}{\sigma_{N,\varphi}^3}
       \sum_{q=1}^k \Big|I_{q-1}\big(f_{q,S}(w,\cdot)\big)\Big|^3.
\end{align*}
Integrating over $w$ and taking expectation yields
\begin{equation}\label{eq:term2-reduction-Yk}
  \mathbb E\int_{\mathbb W} |D_wF_{N,\varphi}|^2\,|D_wL^{-1}F_{N,\varphi}|\,\mu(\mathrm dw)
  \le \frac{C}{\sigma_{N,\varphi}^3}\sum_{q=1}^k T_{q,N},
\end{equation}
where
\[
  T_{q,N}:=\mathbb E\int_{\mathbb W}\Big|I_{q-1}\big(f_{q,S}(w,\cdot)\big)\Big|^3\,\mu(\mathrm dw).
\]
Thus it suffices to show that $T_{q,N}=O(N)$ for each fixed $q\le k$. We will do so by combining a moment bound for multiple integrals
with suitable contraction estimates.

\begin{lemma}[Fourth moment bound via contractions]\label{lem:Iq-fourthmoment-contraction}
Fix an integer $q\ge1$. Let $h\in L_s^2(\mu^{\otimes q})$ be symmetric and assume that, for every
$0\le r\le q$ and $0\le \ell\le r$, the contraction
$h\star_r^\ell h$ is well defined and belongs to
$L^2(\mu^{\otimes(2q-r-\ell)})$. Then there exists
$C:=C(q)\in (0,\infty)$ such that
\begin{equation}\label{eq:Iq-fourthmoment-contraction}
  \mathbb E_\xi\big[I_q(h)^4\big]
  \;\le\;
  C\sum_{r=0}^{q}\sum_{\ell=0}^{r} (2q-r-\ell)!\;
  \big\|h\widetilde{\star}_r^\ell h\big\|_{L^2(\mu^{\otimes(2q-r-\ell)})}^2.
\end{equation}
\end{lemma}

\begin{proof}
By the product formula \eqref{mutiIn} with $q_1=q_2=q$ and $f_1=f_2=h$,
\[
  I_q(h)^2
  =\sum_{r=0}^{q} r!\binom{q}{r}^2\sum_{\ell=0}^{r}\binom{r}{\ell}\,
   I_{2q-r-\ell}\big(h\widetilde\star_r^\ell h\big).
\]
Different pairs $(r,\ell)$ may yield the same chaos order $m$, so the
corresponding multiple integrals need not be orthogonal. We therefore group by chaos orders.
For each $m\in\{0,1,\dots,2q\}$, set
\[
  G_m
  =
  \sum_{\substack{0\le r\le q,\;0\le \ell\le r:\\ 2q-r-\ell=m}}
  \alpha_{r,\ell}\,\big(h\widetilde\star_r^\ell h\big),
  \qquad
  \alpha_{r,\ell}=r!\binom{q}{r}^2\binom{r}{\ell}.
\]
Then $I_q(h)^2=\sum_{m=0}^{2q} I_m(G_m)$. By orthogonality of different chaos orders,
\[
  \mathbb E_\xi\big[I_q(h)^4\big]
  =\sum_{m=0}^{2q}\mathbb E_\xi\big[I_m(G_m)^2\big]
  =\sum_{m=0}^{2q} m!\,\|G_m\|_{L^2(\mu^{\otimes m})}^2.
\]
For each fixed $m$, the sum defining $G_m$ contains only finitely many terms depending on $q$.
Hence, using $\|\sum_{i=1}^M u_i\|_{L^2(\mu^{\otimes m})}^2
\le M\sum_{i=1}^M\|u_i\|_{L^2(\mu^{\otimes m})}^2$ and absorbing the (finite)
combinatorial factors into $C$, we obtain
\[
  m!\,\|G_m\|_{L^2(\mu^{\otimes m})}^2
  \le
  C\sum_{\substack{0\le r\le q,\;0\le\ell\le r:\\ 2q-r-\ell=m}}
  m!\,
  \|h\widetilde\star_r^\ell h\|_{L^2(\mu^{\otimes(2q-r-\ell)})}^2.
\]
Summing over $m$ yields \eqref{eq:Iq-fourthmoment-contraction}.
\end{proof}

\begin{corollary}[Third moment bound via contractions]\label{cor:Iq-thirdmoment-contraction}
Under the assumptions of Lemma~\ref{lem:Iq-fourthmoment-contraction}, there exists $C:=C(q)\in (0,\infty)$ such that
\begin{equation}\label{eq:Iq-thirdmoment-contraction}
  \mathbb E_\xi\big[|I_q(h)|^3\big]
  \le
  C\,\|h\|_{L^2(\mu^{\otimes q})}\,
  \Bigg(\sum_{r=0}^{q}\sum_{\ell=0}^{r}
  \big\|h\widetilde{\star}_r^\ell h\big\|_{L^2(\mu^{\otimes(2q-r-\ell)})}^2\Bigg)^{1/2}.
\end{equation}
\end{corollary}

\begin{proof}
By the Cauchy--Schwarz inequality,
\[
  \mathbb E_\xi\big[|I_q(h)|^3\big]
  \le \mathbb E_\xi\big[I_q(h)^2\big]^{1/2}\,\mathbb E_\xi\big[I_q(h)^4\big]^{1/2}
  =\sqrt{q!}\,\|h\|_{L^2(\mu^{\otimes q})}\,\mathbb E_\xi\big[I_q(h)^4\big]^{1/2}.
\]
Inserting \eqref{eq:Iq-fourthmoment-contraction} and absorbing constants into $C$ yields \eqref{eq:Iq-thirdmoment-contraction}.
\end{proof}

\begin{lemma}\label{lem:TqN-ON-fixed}
Fix $q\in\{1,\dots,k\}$ and write $q'=q-1$. Under \eqref{eq:drift-assumption-Yk}, there exists
$C:=C(q,\lambda,\mathfrak p)\in (0,\infty)$ such that for all $N\ge1$,
\[
  T_{q,N}
  =\mathbb E\int_{\mathbb W}\Big|I_{q'}\big(f_{q,S}(w,\cdot)\big)\Big|^3\,\mu(\mathrm dw)
  \le C\,N.
\]
\end{lemma}

\begin{proof}
We treat $q=1$ and $q\ge2$ separately.

\Case{Case $q=1$.}
Then $q'=0$ and $I_{q'}=I_0$ is the identity. Since $f_{1,S}(w)=\sum_{n=1}^N\mathbf 1_{\mathcal{A}_n}(w)$, the
bound $T_{1,N}=O(N)$ follows as in the low-order case by expanding $f_{1,S}^3$ and using the
summability $\sum_{t\ge1}a_t^{(1)}<\infty$ from Lemma~\ref{lem:moments-decay-Yk}.

\Case{Case $q\ge2$.}
Fix $S$ and $w\in\mathbb W$ and set
\begin{equation}\label{eq:hw-def}
  h_w=f_{q,S}(w,\cdot)\in L_s^2(\mu^{\otimes q'}),\qquad
  h_w(u_1,\dots,u_{q'})=\sum_{n=1}^N \mathbf 1_{\mathcal{A}_n}(w)\prod_{j=1}^{q'} \mathbf 1_{\mathcal{A}_n}(u_j).
\end{equation}
Since $h_w$ is a finite sum of indicator tensors of sets with finite $\mu$-mass, all its
contractions $h_w\widetilde\star_r^\ell h_w$ belong to the required $L^2$ spaces. Here and below, for the scalar contraction corresponding to $(r,\ell)=(q',q')$, we use the convention
$L^2(\mu^{\otimes0})=\mathbb R$.

Applying Corollary~\ref{cor:Iq-thirdmoment-contraction} with $q=q'$ to $h=h_w$ (conditionally on $S$), and writing
$\mathfrak d=2q'-r-\ell$ for the order of the contraction indexed
by $(r,\ell)$,
we obtain
\begin{equation}\label{eq:cond-thirdmoment-hw}
  \mathbb E_\xi\big[|I_{q'}(h_w)|^3\mid S\big]
  \le
  C(q')\,\|h_w\|_{L^2(\mu^{\otimes q'})}\,
  \Bigg(\sum_{r=0}^{q'}\sum_{\ell=0}^{r}\big\|h_w\widetilde\star_r^\ell h_w
\big\|_{L^2(\mu^{\otimes\mathfrak d})}^2\Bigg)^{1/2}.
\end{equation}
Integrating \eqref{eq:cond-thirdmoment-hw} with respect to $\mu(\mathrm dw)$ and using the Cauchy--Schwarz inequality in
$\mu(\mathrm dw)$ yields
\begin{align*}
  \int_{\mathbb W}\mathbb E_\xi\big[|I_{q'}(h_w)|^3\mid S\big]\,\mu(\mathrm dw)
  &\le C(q')\Big(\int_{\mathbb W}\|h_w\|_{L^2(\mu^{\otimes q'})}^2\,\mu(\mathrm dw)\Big)^{1/2}
     \Big(\int_{\mathbb W}\sum_{r,\ell}\|h_w\widetilde\star_r^\ell h_w\|_{L^2(\mu^{\otimes\mathfrak d})}^2\,\mu(\mathrm dw)\Big)^{1/2}.
\end{align*}
Taking $\mathbb E_S$ and using the Cauchy--Schwarz inequality in $\mathbb P_S$ gives
\begin{equation}\label{eq:TqN-reduction}
  T_{q,N}
  \le C(q')
  \Big(\mathbb E\int_{\mathbb W}\|h_w\|_{L^2(\mu^{\otimes q'})}^2\,\mu(\mathrm dw)\Big)^{1/2}
  \Big(\mathbb E\int_{\mathbb W}\sum_{r,\ell}\|h_w\widetilde\star_r^\ell h_w\|_{L^2(\mu^{\otimes\mathfrak d})}^2\,\mu(\mathrm dw)\Big)^{1/2}.
\end{equation}

We first evaluate the factor $\mathbb E\int \|h_w\|_{L^2(\mu^{\otimes q'})}^2\,\mu(\mathrm dw)$. A direct computation using
\eqref{eq:hw-def} yields
\[
  \int_{\mathbb W}\|h_w\|_{L^2(\mu^{\otimes q'})}^2\,\mu(\mathrm dw)
  =\sum_{n,m=1}^N \mu(\mathcal{A}_n\cap \mathcal{A}_m)^{q'+1}
  =\sum_{n,m=1}^N \mu(\mathcal{A}_n\cap \mathcal{A}_m)^{q}.
\]
Since $\mu(\mathcal{A}_n\cap \mathcal{A}_m)=\lambda Q_{|m-n|}(S_m-S_n)$, taking expectation and using stationary
increments we obtain
\[
  \mathbb E\int_{\mathbb W}\|h_w\|_{L^2(\mu^{\otimes q'})}^2\,\mu(\mathrm dw)
  =\lambda^q\Big(N+2\sum_{t=1}^{N-1}(N-t)\,a_t^{(q)}\Big)=O(N),
\]
because $\sum_{t\ge1}a_t^{(q)}<\infty$ by Lemma~\ref{lem:moments-decay-Yk}.

It remains to bound $\mathbb E\int \|h_w\widetilde\star_r^\ell h_w\|_{L^2(\mu^{\otimes\mathfrak d})}^2\,\mu(\mathrm dw)$ uniformly over
$0\le \ell\le r\le q'$. Since symmetrization is an average over permutations, we have
$\|h_w\widetilde\star_r^\ell h_w\|\le \|h_w\star_r^\ell h_w\|$, so it suffices to consider the
unsymmetrized contraction. Using \eqref{eq:hw-def}, bilinearity of $\star_r^\ell$, and the explicit
contraction of indicator tensors (as in Lemma~\ref{lem:Hrs-L2-Yk}), one obtains
\[
  h_w\star_r^\ell h_w
  =\sum_{n,m=1}^N \mathbf 1_{\mathcal{A}_n\cap \mathcal{A}_m}(w)\,\mu(\mathcal{A}_n\cap \mathcal{A}_m)^\ell\,
   \mathbf 1_{\mathcal{A}_n\cap \mathcal{A}_m}^{\otimes(r-\ell)}\otimes \mathbf 1_{\mathcal{A}_n}^{\otimes(q'-r)}
   \otimes \mathbf 1_{\mathcal{A}_m}^{\otimes(q'-r)}.
\]
Expanding the $L^2$-norm squared, integrating in $w$, and using the product structure of
$\mu^{\otimes\mathfrak d}$ gives
\begin{align}\label{eq:hw-contraction-L2-integrated}
  \int_{\mathbb W}\|h_w\star_r^\ell h_w\|_{L^2(\mu^{\otimes\mathfrak d})}^2\,\mu(\mathrm dw)
  &=
  \sum_{n,m,n',m'=1}^N
  \mu(\mathcal{A}_n\cap \mathcal{A}_m)^\ell\,\mu(\mathcal{A}_{n'}\cap \mathcal{A}_{m'})^\ell \\
  &\quad\times
  \mu(\mathcal{A}_n\cap \mathcal{A}_m\cap \mathcal{A}_{n'}\cap \mathcal{A}_{m'})^{\,r-\ell+1}\,
  \mu(\mathcal{A}_n\cap \mathcal{A}_{n'})^{\,q'-r}\,
  \mu(\mathcal{A}_m\cap \mathcal{A}_{m'})^{\,q'-r}.
  \notag
\end{align}
Each summand in \eqref{eq:hw-contraction-L2-integrated} equals
$\mathcal G_{\boldsymbol\gamma}(n,m,n',m')$ from
Lemma~\ref{lem:four-time-connected-Yk}, with $K=q$ and
\[
  (\gamma_{12},\gamma_{34},\gamma_{1234},\gamma_{13},\gamma_{24})
  =(\ell,\ell,r-\ell+1,q'-r,q'-r).
\]
All exponents lie in $\{0,\ldots,q\}$, and, crucially,
$\gamma_{1234}=r-\ell+1\ge1$ for every $0\le\ell\le r\le q'$.
Thus the first condition of Lemma~\ref{lem:four-time-connected-Yk} always holds, including for the
scalar contraction $(r,\ell)=(q',q')$. Consequently,
\[
  \mathbb E_S\int_{\mathbb W}
  \|h_w\star_r^\ell h_w\|_{L^2(\mu^{\otimes\mathfrak d})}^2\,\mu(\mathrm dw)\le CN,
\]
uniformly over $(r,\ell)$. The same bound holds for the symmetrized
contraction. Since the number of pairs $(r,\ell)$ depends only on $q$,
combining this estimate with the first-factor bound in
\eqref{eq:TqN-reduction} gives $T_{q,N}\le CN$.
\end{proof}
Combining \eqref{eq:term2-reduction-Yk} with Lemma~\ref{lem:TqN-ON-fixed} and Lemma~\ref{lem:Yk-Var}, we get
\begin{equation}\label{eq:term2-final-Yk}
 \mathrm T_2^{(\varphi)}
  \le \frac{C}{\sigma_{N,\varphi}^3}\sum_{q=1}^k T_{q,N}
  \le \frac{C N}{\sigma_{N,\varphi}^3}\le C\,N^{-1/2}.
\end{equation}

Finally, combining \eqref{eq:term1-final-Yk} and
\eqref{eq:term2-final-Yk} with \eqref{eq:MS-Yk}, we obtain
\[
  \mathrm T_1^{(\varphi)}\le C N^{-1/2},
  \qquad
  \mathrm T_2^{(\varphi)}\le C N^{-1/2}.
\]
Hence, by \eqref{eq:MS-Yk},
\[
  d_W(F_{N,\varphi},Z)
  \le
  \sqrt{\frac2\pi}\,\mathrm T_1^{(\varphi)}+\mathrm T_2^{(\varphi)}
  \le C N^{-1/2}.
\]
This completes the proof of
Theorem~\ref{thm:Yk-main}.

\section*{Acknowledgements}
We thank the anonymous referee for his/her careful reading of the manuscript and his/her insightful comments and suggestions, which have significantly improved the quality of this paper. Rang was partially supported by the National Natural Science Foundation of China (Nos.~12131019 and 11971361). Su was partially
supported by the National Natural Science Foundation of China (Nos.~12271475 and U23A2064).

\appendix

\section{The symmetric sampling case \texorpdfstring{$\mathfrak p=\tfrac12$}{p = 1/2}}
\label{app:symmetric-sampling}
A natural question is whether the approach developed in Section~\ref{sec:Yk} continues to apply when the
sampling walk is symmetric, that is, when $\mathfrak p=\tfrac12$.

The purpose of this appendix is to clarify this point. The structural ingredients of our proof in the
drifted case remain available under symmetric sampling: the conditional Poisson-chaos expansion and the annealed Malliavin--Stein framework both continue to hold
without essential change. The difference arises at the level of the quantitative estimates. In the drifted regime, the ballistic behavior of the sampling walk yields summable correlation coefficients,
which is precisely what allows the first Malliavin--Stein term to be shown to decay to zero. In the
symmetric regime, by contrast, these correlations decay only polynomially, and the same argument no
longer yields a decaying bound for that term. This should not be
interpreted as indicating that a central limit theorem fails. It only shows that the quantitative
route used in Section~\ref{sec:Yk} does not, in its present form, yield a decaying bound. 
Nevertheless, the symmetric case exhibits a nontrivial and
interesting variance structure, governed by the first non-vanishing Poisson-chaos coefficient of
the observable, and it is therefore worthwhile to record the corresponding calculations explicitly.
We do so below. The question of whether one can nevertheless establish an annealed central limit
theorem, or obtain quenched Wasserstein bounds, under symmetric sampling is left for future work.

In this appendix we discuss the symmetric sampling regime $\mathfrak p=\tfrac12$.
To distinguish it from the drifted sampling walk $S$ used in Section~\ref{sec:Yk}, we denote by
$\widehat S=(\widehat S_n,\,n\in \mathbb N_0)$
an independent SSRW on $\mathbb Z$.
Since $\widehat S$ has the same law as the particle walk, its transition kernel coincides with the
particle kernel, and we continue to write
\[
  Q_n(x)=\mathbb P(\widehat S_n=x),\qquad x\in\mathbb Z,\ n\in\mathbb N_0,
\]
for this common SSRW kernel.

For a fixed polynomial observable $\varphi:\mathbb N_0\to\mathbb R$,
we define the symmetric-sampling path functional
\[
  \widehat Y_{N,\varphi}
  =\sum_{n=1}^N \varphi\big(\xi(n,\widehat S_n)\big),
  \qquad
  \widehat\sigma_{N,\varphi}^2=\operatorname{Var}\!\big(\widehat Y_{N,\varphi}\big).
\]

%==============================================================================
\subsection{\texorpdfstring{$\ell^s$}{ell\textasciicircum s}-norms of the SSRW kernel}
\label{app:lp-kernel}
%==============================================================================

From Lemma~\ref{lem:heat-kernel-rigorous}, we have
\begin{equation}\label{eq:LCLT-periodic-app-equivalent}
  Q_n(x)
  =
  \mathbf 1_{\{x\equiv n\!\!\!\!\pmod 2\}}
  \frac{2}{\sqrt{2\pi n}}
  \exp\!\Big(-\frac{x^2}{2n}\Big)
  +o(n^{-1/2}),
\end{equation}
uniformly in $x\in\mathbb Z$.

We now deduce the asymptotics of the $\ell^s$-norms of $Q_n$.

\begin{lemma}[$\ell^s$-asymptotics for the SSRW kernel]\label{lem:lp-Qn-app}
Fix an integer $s\ge2$. Then, as $n\to\infty$,
\begin{equation}\label{eq:lp-asymp-app}
  \sum_{x\in\mathbb Z} Q_n(x)^s
  = \kappa_s\,n^{-(s-1)/2}+o\!\big(n^{-(s-1)/2}\big),
  \qquad
  \kappa_s:=\frac{2^{(s-1)/2}}{\pi^{(s-1)/2}\sqrt{s}}.
\end{equation}
In particular,
\[
  \sum_{x\in\mathbb Z}Q_n(x)^2\sim \frac{1}{\sqrt{\pi n}},
  \qquad
  \sum_{x\in\mathbb Z}Q_n(x)^3\sim \frac{2}{\pi\sqrt{3}}\,\frac{1}{n}.
\]
\end{lemma}

\begin{proof}
Fix $s\ge2$. Let
\[
  g_n(x)=
  \mathbf 1_{\{x\equiv n\!\!\!\!\pmod 2\}}
  \frac{2}{\sqrt{2\pi n}}
  \exp\!\Big(-\frac{x^2}{2n}\Big).
\]
By \eqref{eq:LCLT-periodic-app-equivalent}, there exists a deterministic sequence
$\varepsilon_n\downarrow0$ such that
\[
  |Q_n(x)-g_n(x)|\le \varepsilon_n n^{-1/2}
  \qquad\text{for all }x\in\mathbb Z.
\]

Fix $M>0$ and split the sum into the central region $|x|\le M\sqrt n$ and the tail
$|x|>M\sqrt n$.

On the central region, uniformly in $|x|\le M\sqrt n$, both $Q_n(x)$ and $g_n(x)$ are of order
$n^{-1/2}$. Hence, by the mean-value theorem,
\[
  |Q_n(x)^s-g_n(x)^s|
  \le C_s\big(Q_n(x)^{s-1}+g_n(x)^{s-1}\big)\,|Q_n(x)-g_n(x)|
  =o(n^{-s/2})
\]
uniformly for $|x|\le M\sqrt n$. Since the number of lattice points in this region is $O(\sqrt n)$,
we obtain
\begin{equation}\label{eq:lp-central-app}
  \sum_{\substack{|x|\le M\sqrt n\\ x\equiv n\!\!\!\!\pmod 2}}Q_n(x)^s
  =
  \sum_{\substack{|x|\le M\sqrt n\\ x\equiv n\!\!\!\!\pmod 2}}g_n(x)^s
  +o\!\big(n^{-(s-1)/2}\big).
\end{equation}

For the tail, the standard heat-kernel bound and Hoeffding's inequality yield
\begin{align}
  \sum_{|x|>M\sqrt n}Q_n(x)^s
  &\le \big(\sup_x Q_n(x)\big)^{s-1}\sum_{|x|>M\sqrt n}Q_n(x)\notag\\
  &\le C\,n^{-(s-1)/2}\,\mathbb P(|\widehat S_n|>M\sqrt n)\notag\\
  &\le C\,n^{-(s-1)/2}\mathrm{e}^{-cM^2}.
  \label{eq:lp-tail-app}
\end{align}
An analogous estimate holds for $g_n$:
\begin{equation}\label{eq:lp-tail-gn-app}
  \sum_{|x|>M\sqrt n}g_n(x)^s
  \le C\,n^{-(s-1)/2}\mathrm{e}^{-cM^2}.
\end{equation}

It remains to evaluate the main term. Since summation over one parity class has mesh size $2$, we have
\begin{align*}
  \sum_{\substack{x\in\mathbb Z\\ x\equiv n\!\!\!\!\pmod 2}}g_n(x)^s
  &=
  \Big(\frac{2}{\sqrt{2\pi n}}\Big)^s
  \sum_{\substack{x\in\mathbb Z\\ x\equiv n\!\!\!\!\pmod 2}}
  \exp\!\Big(-\frac{s x^2}{2n}\Big) \\
  &=
  \Big(\frac{2}{\sqrt{2\pi n}}\Big)^s
  \left(
    \frac12\int_{\mathbb R}\exp\!\Big(-\frac{s u^2}{2n}\Big)\,du
    +o(\sqrt n)
  \right) \\
  &=
  \Big(\frac{2}{\sqrt{2\pi n}}\Big)^s
  \left(
    \frac12\sqrt{\frac{2\pi n}{s}}+o(\sqrt n)
  \right) \\
  &=
  \kappa_s\,n^{-(s-1)/2}+o\!\big(n^{-(s-1)/2}\big).
\end{align*}
Combining this with \eqref{eq:lp-central-app}, \eqref{eq:lp-tail-app}, and
\eqref{eq:lp-tail-gn-app}, and then letting $M\to\infty$, proves \eqref{eq:lp-asymp-app}.
\end{proof}

For $l\ge1$, we will repeatedly use the shorthand
\begin{equation}\label{eq:ahat-ell-symm-def}
  \widehat a_t^{(l)}
  =\mathbb E\big[Q_t(\widehat S_t)^l\big]
  =\sum_{x\in\mathbb Z}Q_t(x)^{l+1}.
\end{equation}
By Lemma~\ref{lem:lp-Qn-app},
\begin{equation}\label{eq:ahat-ell-symm-asymp}
  \widehat a_t^{(l)}
  \sim \kappa_{l+1}\,t^{-l/2},
  \qquad t\to\infty.
\end{equation}

%==============================================================================
\subsection{Variance regimes for general polynomial observables}
\label{app:poly-symmetric}
%==============================================================================
Recall from Lemma~\ref{lem:single-site-k-chaos} that the centered single-site observable admits the
Poisson--Charlier expansion
\[
  \varphi(N_\mathcal{A})-\mathbb E[\varphi(N_\mathcal{A})]
  =\sum_{l=1}^k c_{\varphi,l}\,C_l(N_\mathcal{A};\lambda),
  \qquad N_\mathcal{A}\sim\mathrm{Poi}(\lambda).
\]

The covariance computation from Section~\ref{sec:Yk} continues to hold verbatim under symmetric
sampling; only the subsequent summability analysis changes. In particular, we have the following
exact variance representation.

\begin{lemma}[Variance representation under symmetric sampling]\label{lem:variance-representation-app}
For
\[
  \widehat Y_{N,\varphi}=\sum_{n=1}^N \varphi(\xi(n,\widehat S_n)),
\]
we have
\begin{equation}\label{eq:variance-representation-app}
  \operatorname{Var}\!\big(\widehat Y_{N,\varphi}\big)
  =
  \sum_{l=1}^k c_{\varphi,l}^2\,l!\,\lambda^l
  \Bigg(
    N+2\sum_{t=1}^{N-1}(N-t)\,\widehat a_t^{(l)}
  \Bigg),
\end{equation}
where $\widehat a_t^{(l)}$ is given by \eqref{eq:ahat-ell-symm-def}.
\end{lemma}

\begin{proof}
This is exactly the variance formula proved in Lemma~\ref{lem:Yk-Var}, with $\widehat a_t^{(l)}=\mathbb E[Q_t(\widehat S_t)^l]$.

Under $\mathfrak p=\tfrac12$, the law of $\widehat S_t$ is $Q_t(\cdot)$, and hence
\[
  \widehat a_t^{(l)}
  =\sum_{x\in\mathbb Z}Q_t(x)\,Q_t(x)^l
  =\sum_{x\in\mathbb Z}Q_t(x)^{l+1}.
\]
This proves \eqref{eq:variance-representation-app}.
\end{proof}

We now derive the three variance regimes. Recall the elementary asymptotics
\begin{align}
  \sum_{t=1}^{N-1}(N-t)t^{-1/2}&=\frac43\,N^{3/2}+O(N), \label{eq:summation-half-app}\\
  \sum_{t=1}^{N-1}(N-t)t^{-1}&=N\log N+O(N), \label{eq:summation-one-app}
\end{align}
and, for every $\alpha>1$,
\begin{equation}\label{eq:summation-alpha-app}
  \sum_{t=1}^{N-1}(N-t)t^{-\alpha}\asymp N.
\end{equation}

\begin{lemma}[Variance regimes by Charlier rank]\label{lem:variance-regimes-app}
Let $r_\varphi$ be as in \eqref{eq:Poisson-Charlier-rank}.
\begin{enumerate}
\item If $r_\varphi=1$, then
\begin{equation}\label{eq:rank1-variance-app}
  \operatorname{Var}\!\big(\widehat Y_{N,\varphi}\big)
  \sim
  \frac{8\,c_{\varphi,1}^2\,\lambda}{3\sqrt{\pi}}\,N^{3/2}.
\end{equation}

\item If $r_\varphi=2$, then
\begin{equation}\label{eq:rank2-variance-app}
  \operatorname{Var}\!\big(\widehat Y_{N,\varphi}\big)
  \sim
  \frac{8\,c_{\varphi,2}^2\,\lambda^2}{\pi\sqrt{3}}\,N\log N.
\end{equation}

\item If $r_\varphi\ge3$, then
\begin{equation}\label{eq:rank3plus-variance-app}
  \operatorname{Var}\!\big(\widehat Y_{N,\varphi}\big)\asymp N.
\end{equation}
\end{enumerate}
\end{lemma}

\begin{proof}
Insert \eqref{eq:ahat-ell-symm-asymp} into \eqref{eq:variance-representation-app}.

If $r_\varphi=1$, then
\[
  \widehat a_t^{(1)}\sim \kappa_2 t^{-1/2},
  \qquad \kappa_2=\frac1{\sqrt{\pi}}.
\]
Hence, by \eqref{eq:summation-half-app},
\[
  2c_{\varphi,1}^2\lambda\sum_{t=1}^{N-1}(N-t)\widehat a_t^{(1)}
  \sim
  2c_{\varphi,1}^2\lambda\cdot \frac1{\sqrt{\pi}}\cdot \frac43 N^{3/2}
  =
  \frac{8c_{\varphi,1}^2\lambda}{3\sqrt{\pi}}N^{3/2}.
\]
The total contribution of the terms with $l\ge2$ is $O(N\log N)$
and hence $o(N^{3/2})$, proving \eqref{eq:rank1-variance-app}.

If $r_\varphi=2$, then $c_{\varphi,1}=0$ and
\[
  \widehat a_t^{(2)}\sim \kappa_3 t^{-1},
  \qquad
  \kappa_3=\frac{2}{\pi\sqrt{3}}.
\]
Therefore, by \eqref{eq:summation-one-app},
\[
  2\,c_{\varphi,2}^2\,2!\,\lambda^2\sum_{t=1}^{N-1}(N-t)\widehat a_t^{(2)}
  \sim
  4c_{\varphi,2}^2\lambda^2\cdot \frac{2}{\pi\sqrt{3}}\,N\log N
  =
  \frac{8c_{\varphi,2}^2\lambda^2}{\pi\sqrt{3}}\,N\log N.
\]
The total contribution of the terms with $l\ge3$ is $O(N)$
and hence $o(N\log N)$, proving \eqref{eq:rank2-variance-app}.

Finally, if $r_\varphi\ge3$, then for every $l\ge r_\varphi$ one has
$\widehat a_t^{(l)}\asymp t^{-l/2}$ with $l/2>1$. Hence every summand in
\eqref{eq:variance-representation-app} is $O(N)$ by \eqref{eq:summation-alpha-app}.
For the lower bound, the diagonal contribution already gives
\[
  \operatorname{Var}\!\big(\widehat Y_{N,\varphi}\big)
  \ge c_{\varphi,r_\varphi}^2\,r_\varphi!\,\lambda^{r_\varphi}\,N.
\]
Thus \eqref{eq:rank3plus-variance-app} follows.
\end{proof}

\begin{remark}\label{rem:variance-regimes-app}
The symmetric regime is therefore qualitatively different from the drifted regime:
\begin{itemize}
\item rank-one observables have superlinear variance of order $N^{3/2}$;
\item rank-two observables have variance of order $N\log N$;
\item only from rank three onward does one recover linear variance growth.
\end{itemize}
In particular, the variance behavior under symmetric sampling depends on the first non-zero
Poisson--Charlier coefficient, not merely on the degree of the polynomial.
\end{remark}

%==============================================================================
\subsection{Why the conditional-variance step from the drifted proof no longer closes in the rank-one case}
\label{app:obstruction-rank1}
%==============================================================================

We now explain why the conditional-variance step used in the proof of
Theorem~\ref{thm:Yk-main} no longer yields a decaying bound in the symmetric rank-one regime.

For $l\ge1$, define
\[
  \widehat\Sigma_N^{(l)}
  =\sum_{1\le n<m\le N}Q_{m-n}(\widehat S_m-\widehat S_n)^l.
\]
Then, as in \eqref{eq:VarYk-condS}, the conditional variance may be written as
\begin{equation}\label{eq:Vhat-general-app}
  \widehat V_N
  :=\operatorname{Var}_\xi\!\big(\widehat Y_{N,\varphi}\mid \widehat S\big)
  =
  \sum_{l=1}^k c_{\varphi,l}^2\,l!\,\lambda^l
  \big(N+2\widehat\Sigma_N^{(l)}\big).
\end{equation}

For the rank-one obstruction below, it is enough to record the following crude fluctuation bound for
the leading term $\widehat\Sigma_N^{(1)}$.

\begin{lemma}[A crude fluctuation bound for the rank-one term]\label{lem:VarSigma1-symmetric-app}
For every fixed $l\ge1$, there exists a finite positive constant $C_l$ such that \begin{equation}\label{eq:VarSigma-symmetric-app} \operatorname{Var}\!\big(\widehat\Sigma_N^{(l)}\big)\le C_l\,N^3, \qquad N\ge1. \end{equation} \end{lemma} 
\begin{proof} 
Set \[ X_{n,r}^{(l)}=Q_r(\widehat S_{n+r}-\widehat S_n)^l, \qquad U_r^{(l)}=\sum_{n=1}^{N-r}X_{n,r}^{(l)}, \] so that $\widehat\Sigma_N^{(l)}=\sum_{r=1}^{N-1}U_r^{(l)}$. 
As in Lemma~\ref{lem:VarSigma-casek}, $X_{n,r}^{(l)}$ is measurable with respect to the $\sigma$-field generated by the increments indexed by
$I(n,r):=\{n+1,\dots,n+r\}$.
Hence $\operatorname{Cov}(X_{n,r}^{(l)},X_{m,s}^{(l)})=0$ whenever $I(n,r)\cap I(m,s)=\emptyset$. For fixed $r,s$, the number of overlapping pairs $(n,m)$ is at most $N(r+s)$. Let \[ b_r^{(l)}=\mathbb E\big[(X_{1,r}^{(l)})^2\big] =\mathbb E\big[Q_r(\widehat S_r)^{2l}\big] =\sum_{x\in\mathbb Z}Q_r(x)^{2l+1}. \] Since $2l+1\ge3$, Lemma~\ref{lem:lp-Qn-app} implies $b_r^{(l)}\asymp r^{-l}$, and in particular $b_r^{(l)}\lesssim r^{-1}$. Therefore, by the Cauchy--Schwarz inequality, \[ |\operatorname{Cov}(X_{n,r}^{(l)},X_{m,s}^{(l)})| \le \sqrt{b_r^{(l)}b_s^{(l)}} \lesssim (rs)^{-1/2}. \] Summing over overlapping pairs gives \[ |\operatorname{Cov}(U_r^{(l)},U_s^{(l)})| \le C N(r+s)(rs)^{-1/2}. \] Hence \[ \operatorname{Var}(\widehat\Sigma_N^{(l)}) \le \sum_{r,s=1}^{N-1}|\operatorname{Cov}(U_r^{(l)},U_s^{(l)})| \le C N\sum_{r,s=1}^{N-1}\frac{r+s}{\sqrt{rs}} \le C_l N^3, \] because \[ \sum_{r=1}^{N}r^{1/2}\asymp N^{3/2}, \qquad \sum_{r=1}^{N}r^{-1/2}\asymp N^{1/2}. \] This proves \eqref{eq:VarSigma-symmetric-app}. 
\end{proof}

We now state the obstruction in the form relevant for the total-variance step of the first
Malliavin--Stein term.

\begin{lemma}[Rank-one obstruction for the total-variance step]\label{lem:rank1-obstruction-app}
Assume that $r_\varphi=1$, equivalently $c_{\varphi,1}\neq0$. Then
\[
  \operatorname{Var}\!\big(\widehat Y_{N,\varphi}\big)\asymp N^{3/2},
\]
and the same crude conditional-variance estimate as in the drifted proof yields only a non-decaying
bound:
\[
  \operatorname{Var}\!\Big(\frac{\widehat V_N}{\widehat\sigma_{N,\varphi}^2}\Big)\le C,
\]
where $C$ depends only on $\varphi$ and $\lambda$.
In particular, this estimate alone does not yield a vanishing contribution from the
conditional-variance term appearing in the total-variance control of the first
Malliavin--Stein quantity.
\end{lemma}

\begin{proof}
By Lemma~\ref{lem:variance-regimes-app}, if $r_\varphi=1$ then
\[
  \widehat\sigma_{N,\varphi}^2\asymp N^{3/2},
  \qquad
  \widehat\sigma_{N,\varphi}^4\asymp N^3.
\]
On the other hand, by \eqref{eq:Vhat-general-app} and a crude variance bound for the finite linear
combination defining $\widehat V_N$, together with Lemma~\ref{lem:VarSigma1-symmetric-app}, we obtain
\[
  \operatorname{Var}(\widehat V_N)\le C N^3,
\]
for some constant $C$ depending only on $\varphi$ and $\lambda$. Therefore,
\[
  \operatorname{Var}\!\Big(\frac{\widehat V_N}{\widehat\sigma_{N,\varphi}^2}\Big)
  =
  \frac{\operatorname{Var}(\widehat V_N)}{\widehat\sigma_{N,\varphi}^4}
  \le C.
\]
Thus the crude conditional-variance part of the total-variance estimate no longer yields a decaying
bound in the rank-one symmetric regime.
\end{proof}

\begin{remark}\label{rem:not-no-clt-app}
Lemma~\ref{lem:rank1-obstruction-app} does not prove that a central limit theorem fails under
symmetric sampling. It only shows that the conditional-variance step used in
Section~\ref{sec:Yk} no longer yields a decaying estimate in the rank-one symmetric regime.
\end{remark}

\end{document}